\documentclass[options]{asl}

\usepackage{stmaryrd}
\usepackage[usenames,dvipsnames]{color}
\usepackage{url}
\usepackage{graphicx}
\usepackage{undertilde}
\usepackage{stmaryrd}
\usepackage{relsize}

\title[Demuth's path to randomness]
{Demuth's path to randomness}

\author{Anton\'in Ku\v{c}era}
\revauthor{Ku\v{c}era, Anton\'in}

\address{Faculty of Mathematics and Physics, Charles University in Prague, Prague, Czech Republic}

\email{kucera@mbox.ms.mff.cuni.cz}

\author{Andr\'e Nies}
\revauthor{Nies, Andr\'e}
\address{Department of Computer Science, University of Auckland, Auckland, New Zealand}

\email{andre@cs.auckland.ac.nz}

\author{Christopher P.\ Porter}
\revauthor{Porter, Christopher P.}

\address{Department of Mathematics, University of Florida,
Gainesville, FL 32611-8105, USA}

\email{cp@cpporter.com}

\thanks{Nies is partially  supported by the Marsden Foundation of New Zealand under grant 09-UOA-184. Porter is supported by the National Science Foundation under grant OISE-1159158 as part of the International Research Fellowship Program.}
\newpage

\newtheorem{Theorem}{Theorem}[section]
\newtheorem{theorem}{Theorem}[section]

\newtheorem{proposition}[Theorem]{Proposition}
\newtheorem{cor}[Theorem]{Corollary}

\newtheorem{lemma}[Theorem]{Lemma}

\theoremstyle{definition}
\newtheorem{definition}[Theorem]{Definition}

\newtheorem{remark}[Theorem]{Remark}

\newcommand{\DII}{\Delta^0_2}
\newcommand{\NN}{{\mathbb{N}}}
\newcommand{\RR}{{\mathbb{R}}}
\newcommand{\QQ}{{\mathbb{Q}}}

\newcommand{\sub}{\subseteq}
\newcommand{\sN}[1]{_{#1\in \NN}}

\newcommand{\uhr}[1]{\! \upharpoonright_{#1}}
\newcommand{\ML}{Martin-L{\"o}f}
\newcommand{\SI}[1]{\Sigma^0_{#1}}
\newcommand{\PI}[1]{\Pi^0_{#1}}

\newcommand{\bi}{\begin{itemize}}
\newcommand{\ei}{\end{itemize}}
\newcommand{\bc}{\begin{center}}
\newcommand{\ec}{\end{center}}

\newcommand{\Halt}{{\ES'}}
\newcommand{\ES}{\emptyset}

\newcommand{\ria}{\rightarrow}
\newcommand{\tp}[1]{2^{#1}}
\newcommand{\ex}{\exists}
\newcommand{\fa}{\forall}

\newcommand{\la}{\langle}
\newcommand{\ra}{\rangle}
\newcommand{\Kuc}{Ku{\v c}era}

\newcommand{\fao}[1]{\forall #1 \, }

\newcommand{\Opcl}[1]{[#1]^\prec}

\newcommand{\n}{\noindent}

\newcommand{\vsps}{\vspace{3pt}}

\newcommand{\vsp}{\vspace{6pt}}
\newcommand{\leb}{\mathbf{\lambda}}
\newcommand{\lwtt}{\le_{\mathrm{wtt}}}

\newcommand{\sss}{\sigma}

\newcommand{\lland}{\, \land \, }

\newcommand\+[1]{\mathcal{#1}}

\newcommand{\Op}{{\mathit{Op}}}

\newcommand{\ol}{\overline}
\newcommand{\ul}{\underline}

\newcommand{\DA}{\downarrow}

\def\uh{\upharpoonright}

\newcommand{\russ}{\mathsf{RUSS}}
\newcommand{\bish}{\mathsf{BISH}}
\renewcommand{\int}{\mathsf{INT}}

\begin{document}

\maketitle

\begin{abstract} 
Osvald Demuth (1936--1988) studied constructive analysis from the viewpoint of the Russian school of constructive mathematics. In the course of his work he introduced various notions of effective null set which, when phrased in classical language, yield a number of  major algorithmic randomness notions.  In addition, he proved several results connecting constructive analysis and randomness that were rediscovered only much later.  

In this paper, we trace the path that took Demuth from his constructivist roots to his deep and innovative work on the interactions between constructive analysis, algorithmic randomness, and computability theory.  We will focus specifically on (i) Demuth's work on the differentiability of Markov computable functions and his study of constructive versions of the Denjoy alternative, (ii) Demuth's independent discovery of the main notions of algorithmic randomness, as well as the development of Demuth randomness, and (iii) the interactions of truth-table reducibility, algorithmic randomness, and semigenericity in Demuth's work.

%We give an overview in mostly chronological order. We sketch a proof that Demuth's  notion of Denjoy sets (or reals) coincides with computable randomness. We show that he  worked with a test notion that is  equivalent to Schnorr tests relative to the halting problem.  We also discuss the invention of Demuth randomness, and Demuth's and \Kuc's work on semigenericity.
\end{abstract}
 
 \tableofcontents
 
\section{Introducing Demuth}

 The mathematician Osvald Demuth  worked primarily on 
constructive analysis in the Russian style, which was initiated by Markov, \v{S}anin, Ce{\u{\i}}tin, and others in the 1950s. Born in 1936 in Prague, Demuth graduated from the Faculty of Mathematics and Physics at Charles University in Prague in 1959 with the  equivalent of a master's degree. 
Thereafter he studied  constructive mathematics under the  supervision of A.A.\ Markov Jr. in Moscow, where he successfully defended his doctoral thesis in 1964.  

After completing his doctoral studies with Markov, he returned to  Charles University, completing his Habilitation in 1968.  He remained at Charles University until the end of his life in 1988.  During this period of time, he faced intense persecution for his political views.  After the Russian invasion of the Czech Republic in 1968, Demuth left the Communist Party in 1969, as he was opposed to  the invasion.  The consequences of this decision for Demuth's career were dire.  From 1972-1978, he was forbidden to lecture at the university, although he continued his scientific work during this period.  Further, he was not permitted to travel abroad until 1987.  Lastly, he never achieved the rank of full professor, even though he clearly deserved this rank.

Despite these hardships, Demuth made a number of contributions to constructive analysis.  His most significant work revealed deep and interesting connections between notions of typicality naturally occurring in constructive analysis and various notions of algorithmic randomness.  These connections have only recently been rediscovered (see, for instance, \cite{Brattka.Miller.ea:nd} and \cite{Bienvenu.Hoelzl.ea:12a}).  He was also extremely productive, publishing nearly 60 research articles, including over 45 articles in the journal Commentationes Mathematicae Universitatis Carolinae during the period from 1968 to 1988.  That journal imposed a page limit of 30 pages  per year on him. Only a small number of Demuth's articles were written in collaboration with others.

Demuth's work, especially its connection to computability theory, has been largely underappreciated.  One goal of this paper is to remedy this situation.  We highlight the path that led Demuth from his initial work in constructive analysis to his later work that drew  heavily upon the techniques of computability theory, work in which notions of algorithmic randomness feature prominently.  As we will see, what is particularly noteworthy about this path is how Demuth's constructivism changed over time:  initially, he worked primarily with constructive objects, but in later work, he considered larger classes of non-constructive objects, such as $\DII$ reals (which he called \emph{pseudo-numbers}), then arithmetical real numbers, and eventually the collection of all real numbers.  Even in this latter phase, however, Demuth did not abandon his constructivist roots, still framing his results in the language of constructive analysis (albeit extended to allow for reference to non-constructive objects).

We also discuss a number of recent developments in algorithmic randomness that can be seen as extending Demuth's results.  We will concentrate in particular on developments that link computable analysis, specifically differentiability and almost everywhere behavior, with notions of randomness. While some of the results we survey are not Demuth's contributions, they fit naturally into his program of studying constructive analysis through the lens of computability theory. We have not discussed similar results that link notions of algorithmic randomness and ergodic theory.

The outline of this paper is as follows.  In $\S$\ref{sec-constructivism}, we will briefly discuss Demuth's constructivism as laid out in his survey ``Remarks on constructive mathematical analysis" \cite{Demuth.Kucera:79}, 
 co-authored with the first author of this paper and published in 1979. In $\S$\ref{sec-Markov}, we will review the basics of Markov computability of real-valued functions. $\S$\ref{sec-randomness} concerns the notions of algorithmic randomness that appear in Demuth's work, especially in his study of the differentiability of Markov computable functions and the Denjoy alternative.  In $\S$\ref{sec-dem-rand} we look closely at Demuth's own notions of randomness, nowadays known as \emph{Demuth randomness} and \emph{weak Demuth randomness}, outlining a number of facts that Demuth proved about these notions as well as some additional results that have been recently obtained.  In $\S$\ref{sec-rstt}, we consider Demuth's work on the interactions of truth-table reducibility, algorithmic randomness, and semigenericity, some of which was carried out jointly with the first author.  In $\S$\ref{sec-concluding}, we conclude with some remarks on Demuth's contributions.

%Before we consider Demuth's constructivism, it is worth noting that
 Interpreting Demuth's results can be a  difficult task that may involve some  guesswork.  This is  evident from a quick look at one of his   papers (see the electronic databases given in $\S$\ref{sec-concluding}), or even at the sample of  his writing  given in  Figure \ref{Fig:DemuthDef} below. The main problem is that the papers  are written in notationally heavy,   formal constructive language. The constructive results have to be re-interpreted classically. So, when we attribute a result to Demuth that was later proved independently, it does not diminish the credit due for this rediscovery the way that it would if Demuth had written his work in the customary classical language of computable analysis.
 
This paper is a substantially extended version of the conference paper \cite{Kucera.Nies:12}. Here we cover Demuth's research more broadly and in more detail.

Part of this work was carried out while \Kuc\ and Nies visited the Institute for Mathematical Sciences at the National University  of Singapore in June 2014. Nies acknowledges support through the Marsden fund of New Zealand.  Porter and \Kuc\ worked on this project at Universit\'e Paris Diderot in October 2013 and at Charles University in Prague in February 2014.  Porter acknowledges support from the National Science Foundation and Charles University.

\section{Demuth's constructivism}\label{sec-constructivism}

Before we survey Demuth's technical results, we will briefly review the basics of the Russian school of constructive mathematics and introduce Demuth's unique approach to constructive mathematics, which is laid out in the 1978 survey paper ``Remarks on constructive mathematical analysis" \cite{Demuth.Kucera:79}, written by Demuth and Ku\v cera.\footnote{We should note that at the time of the publication of~\cite{Demuth.Kucera:79}, its  second author accepted the basic principles of constructivism.  However, in the years that followed, he gradually turned to  the use of classical non-constructive methods.}

Constructive mathematics in the Russian school ($\russ$), like other versions of constructivism, namely Bishop's constructive mathematics ($\bish$) and Brouwer's intuitionism ($\int$), aims to put mathematics on a secure foundation.  Like $\bish$ and $\int$, in $\russ$ one rejects the general use of the law of excluded middle and thus double negation elimination.  However, what distinguishes $\russ$ from $\bish$ and $\int$ is the central role that the notion of effectivity plays in the theory.
 
But why emphasize this notion of effectivity?  According to Demuth and \Kuc{}, there is a historical reason, for as they write \cite[p. 81]{Demuth.Kucera:79},
\begin{quote}
From the historical point of view, the development of mathematics was substantially influenced by applications of mathematics where solutions of problems consisted, de facto, in transformation of particular information coded by words.
\end{quote}
They add, ``The means necessary for algorithmic processing of words are indispensable for any sufficiently rich mathematical theory" \cite[p. 81]{Demuth.Kucera:79}.  Demuth and \Kuc{} further held that such means also prove to be \emph{sufficient} for developing rich mathematical theories.

The distinctive features of $\russ$ as laid out by Demuth are the  following. 
\begin{enumerate}
\item The objects studied are constructive objects, coded as words in a finite alphabet.

\medskip

\item The Church-Turing Thesis is accepted (see \cite{Church:36a}, \cite{Turing:36a}). That is, the algorithms by means of which the words coding constructive objects are transformed are precisely the Turing computable functions, or equivalently the Markov algorithms (which is a formalism often used in the Russian school; see, for instance, \cite{Markov:54}).  

\medskip

\item The so-called ``constructive interpretation of mathematical propositions" as developed by \v{S}anin \cite{Sanin:58} is used.  According to this interpretation, the existential quantifier and disjunction are interpreted constructively.  That is, one is entitled to assert the existence of an object if there is an algorithmic procedure for constructing the object, and one is entitled to accept a disjunction of two formulas if there is an algorithmic procedure for determining which of the disjuncts is true.

\medskip

\item The following principle, known as Markov's Principle, is allowed:  if one has refuted the claim that some Markov algorithm $A$ does not accept a given input $x$, then one can conclude that $A$ accepts $x$.  In modern notation, 
\[
\neg\neg A(x){\DA}\;\Rightarrow \; A(x){\DA},
\]
where $A(x){\DA}$ means that the Markov algorithm $A$ halts on input $x$.  Thus, although double negation elimination is not permissible in general, it can be used in specific situations as laid out by Markov's Principle.

\end{enumerate}
For more details on the basic features of $\russ$, see \cite[Chapter 1]{Kushner:84} or the survey \cite{Kushner:99}.

Though Demuth is explicit about his commitment to the principles of $\russ$, in surveying his work, one will find that Demuth routinely appeals to objects and techniques that are well beyond the scope of what is constructively admissible, at least according to principles accepted by most constructivists.

One finds that Demuth became more and more lenient about the objects to which he appealed in his theorems.  That is, he gradually extended his domain of discourse to include more and more non-constructive objects.  Initially, he restricted his attention to the computable real numbers, but later he formulated certain of his results in terms of the more general collection of $\Delta^0_2$ reals.  Later yet he further extended his work to encompass the collection of arithmetical reals, and finally in his last papers, he even proved statements involving quantification over all real numbers.  How did Demuth account for this use of non-constructive objects?

Demuth's answer to this question is a subtle one.  He  formulated many of his later results in terms of the collection of all real numbers, proving, for example, that a number of computability-theoretic statements hold relative to any oracle. However, he was primarily concerned with these results insofar as they apply to arithmetical real numbers.  As he and \Kuc{} write,
\begin{quote}
It should be noted that we are interested, owing to the natural connection between concepts of constructive mathematical analysis and arithmetical predicates, only in the computability relative to jumps of the empty set. \cite[p. 84]{Demuth.Kucera:79}
\end{quote}
Here Demuth appeals to Post's theorem, according to which a predicate $P$ of the natural numbers is arithmetical if and only if there is some $n\in\NN$ such that $P$ is computable relative to $\emptyset^{(n)}$, the $n$-th Turing jump of the empty set.  Of course, one who strictly adheres to the principles of $\russ$ would find such an appeal to $\emptyset^{(n)}$   completely unacceptable.  

However, in Demuth's view, functions computable from $\emptyset^{(n)}$, where    $n\in\NN$, are still in some sense constructively grounded, as they ``can be represented on the basis of recursive functions by means of non-effective limits" according to the strong form of Shoenfield's Limit Lemma (see, for instance \cite[Corollary 2.6.3]{Downey.Hirschfeldt:book}).  Demuth and \Kuc{} further argue,

\begin{quote}
Without leaving [the] constructive program concerning effective processes we improve, by the use of relative computability, our ability to handle effective procedures.  The advantage of the improvement consists in both substantial simplification and clearness of formulations \cite[p. 84]{Demuth.Kucera:79}.
\end{quote}
Thus, Demuth studied notions connected to relative computability from a constructive point of view.  Although he himself was only concerned with arithmetical reals as potential inputs for algorithmic procedures, he left open the possibility of considering his results in terms of a broader class of inputs, even, potentially, the entire collection of real numbers.\footnote{It is known from private communication with the first author that Demuth would have also accepted hyperarithmetical reals, but he saw no need to work with them.}

Whether or not one agrees that Demuth was still being faithful to the basic principles of $\russ$, it is fair to characterize Demuth's approach as an \emph{extended constructivism}.  As we will see in the sections that follow, this extension was a gradual one, but it allowed him to bring techniques of constructive mathematics and classical computability theory together in interesting and often insightful ways.

\section{The basic definitions of computable analysis in the Rus\-sian school}\label{sec-Markov}

As Demuth worked primarily in the field of computable analysis, we will review the basic definitions of this subject.  In this paper,  these definitions  will be phrased   in the language of modern computable analysis, as developed, for instance, in  Brattka et al.\ \cite{Brattka.Hertling.ea:08},  Pour-El and Richards~\cite{Pour-El.Richards:89} and Weihrauch~\cite{Weihrauch:00}.   See  Aberth~\cite{Aberth:80}  for a more recent discussion of computable analysis in the Russian school.

One of the central notions of computable analysis in the Russian style is the notion of a \emph{constructive} real number. Its  modern analogue   is the  notion of a \emph{computable} real number.

\begin{definition}[Turing \cite{Turing:36a}] \label{def:compreal} {\rm A \emph{computable real number} $z$  is given by a computable Cauchy name,  i.e., a computable  sequence $(q_n)\sN n$ of   rationals converging to $z$ such that $|q_{k} - q_n | \le \tp{-n}$ for each $k \ge n$. } \end{definition} 

A sequence  $(x_n)\sN n$ of reals is computable if there is a computable double sequence $(q_{n,k})\sN {n,k}$ of rationals such that each $x_n$ is a computable real as witnessed by its Cauchy name $(q_{n,k})\sN k$.

As is well-known, one can equivalently define a computable real number in terms of a computable sequence of rationals $(q_n)\sN n$ and a  computable function $f:\mathbb{N}\rightarrow\mathbb{N}$ such that for every $n$ and every $k\geq f(n)$,
$|q_{f(n)}-q_k|\leq 2^{-n}$.  In modern terminology, $f$ is referred to as a \emph{modulus function}, whereas Demuth referred to $f$ as the \emph{regulator of fundamentality} of the sequence $(q_n)_{n\in\mathbb{N}}$.  We will write $\mathbb{R}_c$ to denote the collection of computable real numbers.

We should note one subtle difference between the constructive approach to computable real numbers and the modern approach. In the constructive approach, a computable real number is held to be a finite syntactic object, given by the pair consisting of the index of the sequence $(q_n)\sN n$ and the index of the modulus $f$.  However, in the modern approach, one need not take a computable real number to be some finite object.  Instead, a computable real number is simply a real number that has a computable name. This approach is compatible with a non-constructive view of real numbers, according to which they are completed totalities.

According to the Russian school, the continuum should be understood constructively, in the sense that it consists entirely of computable real numbers.  From this point of view, the continuum should not be seen as having gaps, since the constructive continuum is constructively complete, in the sense that every uniformly computable Cauchy sequence of real numbers with a computable modulus of Cauchy convergence converges to a computable real number (see \cite{Kushner:84}).

The notion of a Markov computable function   was central to Russian-style constructivism.  In the context of the constructive continuum, this is a natural notion of computability for a function.

For a computable Cauchy name $(q_n)\sN n$, if $\phi_i$ is a computable function such that $\phi_i(n)=q_n$ for every $n$, then we call $i$ the \emph{index} of $(q_n)\sN n$.  The following definition is due to Markov \cite{Markov:58}. In keeping with the constructive commitment to studying the transformation of finite words, a Markov computable function can   be seen as uniformly transforming an algorithm for computing a given computable real into another algorithm for computing the output real.

 \begin{definition} \label{def:Markov_computable}  {\rm A function $g:\mathbb{R}_c\rightarrow\mathbb{R}_c$ defined on the computable reals  is called \emph{Markov computable} if from any index for a  computable Cauchy name for $x$ one can compute an index for a  computable  Cauchy name for $g(x)$.}
\end{definition}

Demuth referred to Markov computable functions as \emph{constructive}. By a  $c$-\emph{function} he meant a constructive function that is constant on $(-\infty, 0]$ and on $[1, \infty)$.   This in effect restricts the domain to the unit interval.  Note that a constructivist cannot explicitly write this restriction to [0,1] into the definition since the relation $x\le y$ is not decidable for computable reals and, thus, it is not decidable  whether a given computable real is negative.
Hereafter we will only make reference to Markov computable functions (we will assume when necessary that a given Markov computable function is constant outside of the unit interval).

By a result of Ce{\u{\i}}tin (see, for instance, \cite{Cei56}, \cite{Cei59}, and \cite{Cei62}) and also a similar result of Kreisel, Lacombe and Shoenfield  \cite{KLS59}, 
each Markov computable function is continuous on the computable reals (with respect to the subspace topology on $\mathbb{R}_c$).  However,   since such a function  may only  be  defined   on the computable reals, it is not necessarily uniformly continuous. This was first shown by Zaslavski\v\i\ in \cite{Zaslavskii:62}.

Such an example of a Markov computable function that is not uniformly continuous can be produced by a typical construction in constructive analysis.  In this construction,  a Markov computable function is defined in terms of a $\SI 1$ class $\mathcal A$ that contains all computable reals. In a natural way $\mathcal A$ may be viewed as a c.e.\ set $S$ of rational intervals.
Now one may describe  a Markov computable function on computable reals by defining it on all rational intervals from $S$.  However, in general, for a computable real $z$ we cannot find exactly one interval from $S$ containing $z$. This is due to the fact that the relation $x \leq y$ is not decidable for computable reals and, thus, given an interval $[a,b]$ we cannot in general determine 
whether a given computable real belongs to $[a,b]$.
At best, for a computable real $z$ we can   find  rationals $a < b < c$ such that the  intervals $[a,b]$ and $ [b,c]$ belong to $S$ and $a < z < c$.  Thus, a Markov computable function $f$ has to be defined consistently and continuously 
on computable reals from any open interval $(a,c)$ such that $[a,b],[b,c]$ belong to $S$ for some $b$.

In this way one may construct a Markov computable function $f$ that is continuous on the computable reals but is not uniformly continuous:  Let $S$ be an infinite c.e.\ set of non-overlapping rational intervals with the property that for every computable real $x$ there is some $I\in S$ such that $x\in I$.  Let $(I_n)\sN n$ be an effective enumeration of the intervals in $S$.  We define $f$ to be 
piecewise linear on each interval from $S$, so that for each $n$, $f$ is equal to 0 at the endpoints of $I_n$ and takes its local maximum with value $n$ at the midpoint of $I_n$. 

Now, for any real $r$ not covered by any interval $I\in S$, $f$ takes arbitrarily large values at computable reals sufficiently close to $r$.  Hence $f$ is not uniformly continuous --- we cannot even continuously extend $f$ to any real outside the union  of the intervals in $S$.

The notion of a Markov computable function should be contrasted with the standard definition of a computable real-valued function from modern computable analysis, hereafter, a \emph{standard} computable function, which is essentially due to Turing \cite{Turing:36a} (although Borel had formulated the basic ideas of computability of real-valued functions in \cite{Borel:12}; see \cite{Avigad.Brattka:12} for a helpful discussion of these developments).  In  this approach, $f:\mathbb{R}\rightarrow\mathbb{R}$ is computable if 
\begin{itemize}
\item[(i)] for every computable sequence of real numbers $(x_k)_{k\in\NN}$, the sequence $(f(x_k))_{k\in\NN}$ is computable, and 
\item[(ii)] $f$ is \emph{effectively uniformly continuous}, i.e., there is a computable function $p:\NN\rightarrow\NN$ such that for every $x,y\in\mathbb{R}$ and every $n\in\NN$,
\[
|x-y|\leq 2^{-p(n)}\Rightarrow|f(x)-f(y)|\leq 2^{-n}.
\]
\end{itemize}
Every standard computable function is uniformly continuous, unlike the case with Markov computable functions, as mentioned above.
However, a significant portion of Demuth's work was concerned with uniformly continuous Markov computable functions.  Recall that a \emph{modulus of  uniform continuity} for a  function  $f$ is a function $\theta$ on positive rationals  such that $|x-y| \le \theta(\epsilon) $ implies $|f(x)- f(y) | \le \epsilon $ for each rational $\epsilon >0$. 

From a constructive point of view, it is reasonable to study uniformly continuous Markov computable functions with a \emph{computable} modulus of uniform continuity, which Demuth referred to as \emph{$\emptyset$-uniformly continuous functions}.  Note that the restricton of a standard computable real-valued function to the computable reals yields a $\emptyset$-uniformly continuous Markov computable function (see~\cite{Brattka.Hertling.ea:08,Weihrauch:00}).

Even if we consider uniformly continuous Markov computable functions with a \emph{non-computable} modulus, such a modulus cannot have arbitrarily high complexity.  Demuth proved that every uniformly continuous Markov computable function has a modulus that is computable in $\emptyset'$.  Demuth thus referred to classically uniformly continuous Markov computable functions as \emph{$\emptyset'$-uniformly continuous}.  Demuth proved a more general result about uniformly continuous Markov computable functions.  Before we state the result,  we need one additional definition.

Let $f:\mathbb{R}_c\rightarrow\mathbb{R}_c$ be a Markov computable function.
We define $R[f]:\mathbb{R}\rightarrow\mathbb{R}$ to be the classical function that is the maximal extension of $f$ that is continuous on its domain.  More precisely, for each non-computable $r\in[0,1]$, if $\ell=\lim_{x\rightarrow r}f(x)$ exists, then we set $R[f](r)=\ell$.  Otherwise, $R[f](r)$ is undefined. 

Recall that a real $r$ is $\Delta^0_3$ if and only if $r\leq_T\ES''$, 
i.e., there is a $\ES''$-computable sequence $(q_n)_{\sN{n}}$ of rationals converging to $r$ such that $|q_{k} - q_n | \leq 2^{-n}$ for each $k \geq n$.

\begin{theorem}[Demuth, Kryl, \Kuc{} \cite{Demuth.Kryl.ea:78}, \cite{Demuth:88a}]\label{thm-uniform-continuity}
Let $f$ be a Markov computable function.  Then the following are equivalent.
\begin{enumerate}
\item $f$ is   uniformly continuous.
\item $f$ is $\emptyset'$-uniformly continuous.
\item $R[f]$ is defined at all $\Delta^0_3$ reals.
\item $R[f]$ is defined at all reals.
\end{enumerate}
\end{theorem}

\begin{proof}
($1\Rightarrow2$):  If $f$ is uniformly continuous, then $\ES'$ can compute a modulus of uniform continuity for $f$. 

The implications ($2\Rightarrow3$), ($2\Rightarrow4$), and ($4\Rightarrow3$) are immediate.  It remains to show ($3\Rightarrow1$).

 We claim that if $f$ is not uniformly continuous then there is a $\ES''$-computable real $x$ at which $R[f]$ is not defined, i.e., $f$ cannot be extended continuously to $x$.  For suppose that there is an $n$ such that 
\begin{itemize}
\item[($\diamond$)] for every $k$ there exist $x,y$ with $|x-y|< 2^{-k}$ and $|f(x)-f(y)|> 2^{-n +1}$.
\end{itemize}
For each $\sigma\in\{0,1\}^*$, the \emph{interval represented by $\sigma$}, denoted $[\sigma)$, is defined to be the half-open interval $[0.\sigma,0.\sigma+2^{-|\sigma|})$.
Now since $f$ is defined and continuous on all dyadic rationals the condition ($\diamond$) can be replaced  with the following:
\begin{itemize}
\item[($\diamond'$)] for every $k$ there exist a string $\sigma$ of length $k$ and rationals $x,y$ in the interval represented by $\sigma$ (so that $|x-y|<2^{-k}$) and $|f(x)-f(y)|> 2^{-n}$.
\end{itemize}
Indeed, for $x^*,y^*\in[0,1]$ such that $|x^*-y^*| < 2^{-k}$ and $|f(x^*)-f(y^*)|> 2^{-n +1}$, if $x^*,y^*$ do not belong to an interval $[\sigma)$ 
for some string $\sigma$ of length $k$ then
for some  $j$ such that $0 < j < 2^k,$ we have  $x^* < j/2^{-k} \leq y^*$. But then either $|f(x^*)-f(j/2^{-k})|> 2^{-n}$ or $|f(y^*)-f(j/2^{-k})|> 2^{-n}$. Now using the continuity of $f$ at $j/2^{k}$
we can easily  find $x,y$ for which the condition ($\diamond'$) holds.

We can use condition ($\diamond'$) to build a $\ES'$-computable tree such that $\sigma\in \{0,1\}^*$ is on the tree if and only if there are
$x,y$ which belong to the interval represented by $\sigma$ and  $|f(x)-f(y)|> 2^{-n}$.
By condition ($\diamond'$), this tree is infinite.  Thus, using $\ES''$ as an oracle, we can compute an infinite path through this tree, which corresponds to a real $x$.
At this real $x$, $R[f]$ is not defined.
\end{proof}

Using Theorem \ref{thm-uniform-continuity}, it is not difficult to verify that for every $\emptyset$-uniformly continuous Markov computable function $f$, the function $R[f]$ is a standard computable function. Indeed, since $f$ is $\emptyset$-uniformly continuous, it is clearly $\emptyset'$-uniformly continuous, and so by Theorem \ref{thm-uniform-continuity}, $R[f]$ is defined on all reals.  Since $R[f](x)=f(x)$ for all computable reals, condition (i) in the definition of a standard computable function is satisfied.  Furthermore, condition (ii) in the definition of a standard computable function is also satisfied, as $R[f]$ is $\ES$-uniformly continuous with the same modulus as $f$, since $f$ is $\ES$-uniformly continuous on a dense subset of $\mathbb{R}$.

Thus, just as $\emptyset$-uniformly continuous Markov computable functions can be obtained by restricting standard computable functions to $\mathbb{R}_c$, standard computable functions can be obtained by extending $\emptyset$-uniformly continuous Markov computable functions from $\mathbb{R}_c$ to $\mathbb{R}$ via the operator $R$.

We should note that Theorem \ref{thm-uniform-continuity} was not originally formulated in terms of the operator $R$ that extends a Markov computable function to a classical function but rather in terms of a more restricted operator.  For instance, in \cite{Demuth.Kryl.ea:78} and \cite{Demuth:80}, Demuth defined $\Op[f]$ to be the maximal continuous extension of $f$ that is defined on all \emph{arithmetical} reals.  However, in Demuth's later papers such as \cite{Demuth:88} and \cite{Demuth:88a}, we find the operator $R$ that behaves like $\Op$ except that for a Markov computable function $f$ the domain of the function $R[f]$ can potentially be defined on \emph{all} real numbers.  This is another example of Demuth's willingness to recast his results in terms of non-constructive objects.

%
%First, in  \cite{Demuth.Kucera:79}, an operator $\Op$ is defined such that for a given Markov computable function $f$, $\Op[f]$ is the maximal continuous extension of $f$ that is defined on all $\Delta^0_2$ reals.  However,

%Demuth  still accepted   talking  about $\DII$ reals, which he called pseudo-numbers. They are given as limits of computable sequences of rationals, so it was not necessary to view them as entities of their own. Later on, in the 1980s, he relaxed his standpoint somewhat, also admitting   arithmetical reals.

\section{Notions of randomness in Demuth's work}\label{sec-randomness}
 As discussed in the introduction, Demuth considered a number of different notions of effective null set. They are equivalent to several major randomness notions that have been  introduced independently.

It is striking that Demuth never actually referred to random or non-random sequences. Instead, he characterized these classes in terms of non-approximability in measure and approximability in measure, respectively.   This reflects the fact that Demuth's motivation in introducing these classes differed significantly from the motivation of the recognized ``fathers" of algorithmic randomness.  Whereas the various randomness notions were introduced and developed by Martin-L\"of, Kolmogorov, Levin, Schnorr, Chaitin, and others in the context of classical probability, statistics, and information theory, Demuth developed these notions in the context of and for application in constructive analysis, where the notion of approximability plays a central role.

For the sake of readability, we will review the main definitions of algorithmic randomness that Demuth introduced. We  will refer to them  in the text that follows.  See \cite{Downey.Hirschfeldt:book} or  \cite{Nies:book} for details.  In the following, $\lambda$ denotes the Lebesgue measure.

\begin{itemize}
\item Martin-L\"of randomness (Martin-L\"of \cite{Martin-Lof:66}):  A Martin-L\"of test is a computable sequence of effectively open sets $(G_m) _{m \in \NN}$  such that $\leb (G_m)\leq 2^{-m}$ for every $m$.  A real $x\in[0,1]$ is \emph{Martin-L\"of random} if $x\notin\bigcap_{m\in\NN}G_m$ for every Martin-L\"of test $(G_m) _{m \in \NN}$.   A Solovay test \cite{Solovay:75}  is a computable sequence of effectively open sets $(G_m) _{m \in \NN}$  such that $\sum\leb (G_m)<\infty$.  A real $x$ passes the test  if $x\in G_m$ for at most finitely many $m$.  Solovay proved that a real  passes all Solovay tests if and only if it is Martin-L\"of random (see, e.g., \cite[Theorem 6.2.8]{Downey.Hirschfeldt:book} or \cite[Proposition 3.2.19]{Nies:book}). \\

\item Schnorr randomness (Schnorr \cite{Schnorr:71}):  A Schnorr test is a computable sequence of effectively open sets $(G_m) _{m \in \NN}$  such that (i) $\leb (G_m)\leq 2^{-m}$ for every $m$ and (ii) $\leb (G_m)$ is a computable real uniformly in~$m$. Furthermore, a real $x$ is \emph{Schnorr random} if and only if $z \notin \bigcap_m G_m$ for every Schnorr test $(G_m) _{m \in \NN}$. 
Note that every Schnorr test is a Martin-L\"of test. This implies that every Martin-L\"of random real is Schnorr random.  However, not every Martin-L\"of test is a Schnorr test, as we do not require that $\lambda(G_m)$ be computable in the definition of a Martin-L\"of test.  Moreover, there are Schnorr random reals that are not Martin-L\"of random.\\

\item Computable randomness (Schnorr \cite{Schnorr:71}): A computable martingale is a computable rational-valued function $M:\{0,1\}^*\rightarrow \mathbb{Q}$ such that $2M(\sigma)=M(\sigma0)+M(\sigma1)$ for every $\sigma\in \{0,1\}^*$.  A computable martingale $M$ succeeds on $X\in2^{\NN}$  if $\sup_nM(X\uh n)=\infty$.  We say that  $X\in2^\NN$ is \emph{computably random} if no computable martingale succeeds on $X$.  A  real $x\in[0,1]$ is computably random if the sequence $X$ such that $0.X=x$ is computably random (we assume here $x$ is not a dyadic rational, so  $X$ is unique).  Every Martin-L\"of random real is computably random.  Every computably random real is Schnorr random. Neither of the implications can be reversed.
\end{itemize}

Demuth   considered   notions of randomness other than  the four listed above. These notions  include what are now  known as Demuth randomness and weak Demuth randomness. They will be  introduced in Section \ref{subsec-denjoy-alternative}
and further discussed in Section \ref{sec-dem-rand}.

\subsection{Measurability and randomness}

We first consider the earliest appearance of a randomness notion in Demuth's work, which was in the context of constructive measurability.  

In the   papers \cite{Demuth:69} and \cite{Demuth:69a} published in 1969, Demuth defines what it means for a property to hold for ``almost every" computable real number. Demuth's definition is given in terms of what he calls $S_\sigma$ sets.  A computable sequence $\{H_n\}_{n\in\mathbb{N}}$ of non-overlapping intervals with rational endpoints is an \emph{$S_\sigma$ set} if $\sum_{n\in\mathbb{N}}|H_n|$ is a computable real number.

A property $\+ P$ of computable reals holds for ``almost every computable real" if there exists a computable sequence of $S_\sigma$ sets $(\mathcal{S}_n)_{n\in\NN}$ such that for every $n$, $\lambda(\mathcal{S}_n)\leq 2^{-n}$ and for any computable real $x$, if $x\notin\mathcal{S}_n$ for some $n$, then $x$ satisfies the property $\+ P$.  It is immediate that such a collection $(\mathcal{S}_n)_{n\in\NN}$ is a Schnorr test.

In formulating his definition of a property holding for almost every computable real, Demuth drew on earlier work of Ce\v\i tin and Zaslavski\v\i\ \cite{Zaslavskii.Ceitin:1962} from 1962.  In this context, it is interesting to note that the absence  of a universal Schnorr test follows from a result of Ce\v\i tin and Zaslavski\v\i\ from that paper, where it is proved (in different terminology) that a $\Pi^0_1$ class of computable measure has a computable path.

Demuth later defined what it means for a property $\+ P$ to hold   for ``almost every''  pseudo-number (i.e.,  $\DII$ real) in \cite[page 584]{Demuth:75}.  In \cite{Demuth.Kucera:79}, it is stated that such a definition can be obtained by directly relativizing to $\ES'$  the definition of a property holding for almost every computable real or ``without using relativized concepts."  It is not clear whether he took these two approaches to be equivalent.  The reference Demuth gives for the unrelativized definition contains the definition given in Figure~\ref{Fig:YtestDef}, but he does not state that this definition is equivalent to the relativized definition.

\begin{figure}[hbt] 
	\bc \scalebox{1.0}{\includegraphics{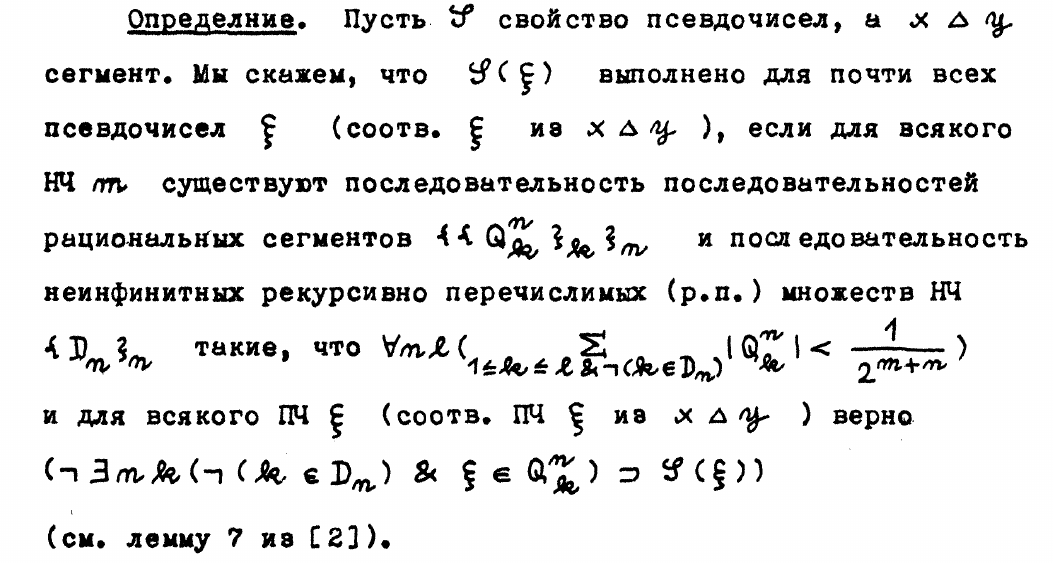}} \ec
		\vspace{-.5cm}
	\caption{\cite[page 584]{Demuth:75}: Definition of interval sequence   tests}  \label{Fig:YtestDef}
\end{figure}

 We  rephrase this definition in modern  language.  
Demuth introduces  a notion  of tests;  let us call them \emph{interval sequence  tests}. In the following let $m,r,k$ range over the set of positive integers.  
An interval sequence  test uniformly  in a number $m\in \NN$ provides a computable   sequence  of rational intervals $(Q^m_r(k))_{r,k \in \NN}$,  and a uniformly c.e.\ sequence of   finite  sets $(E^m_r)_{r\in\NN}$, such that   
\begin{equation}
\label{eqn:StestPart} \leb  \bigl( \bigcup \{Q^m_r(k) \colon \, k \not \in  E^m_r\} \bigr) \le \tp{-(m+r)}.
\end{equation}
The idea here is that each finite set $E^m_r$ consists of indices $\{k_1,\dotsc,k_n\}$ for rational intervals $Q^m_r(k_1),\dotsc,Q^m_r(k_n)$ such that
those $z\in Q^m_r(k)$ for some $k\notin E^m_r$ are contained in a fairly small set, i.e., one with measure less than $2^{-(m+r)}$.

 A real  $z$  \emph{fails}  the interval sequence test if  for each $m$ there is  $r$ such that for some $k  \not \in E^m_r$  we have $z \in   Q^m_r(k)$. In other words, for each $m$,  
\begin{equation}\label{eqn:StestFail}
z \in \bigcup_r \bigcup_{k \not \in  E^m_r}  Q^m_r(k).
\end{equation}
Note that the class in (\ref{eqn:StestFail})  has measure at most $\tp{-m}$, hence the reals~$z$ failing the test form a null set. If $z$ does not fail the test we say that   $z$ \emph{passes} the test.  
   Demuth says that a  property $\+ P $ holds for \emph{almost all} reals~$z$ if there is   an  interval sequence  test (depending on $\+ P$) such that $\+ P $ holds for all $z$ passing the test.  
 
We now show that this unrelativized definition of a property holding for almost every pseudo-number is equivalent to the relativization of Demuth's definition of a property holding for almost all computable reals.
 
 \begin{proposition}[with Hirschfeldt] Interval  sequence tests are uniformly equivalent to Schnorr tests relative to $\Halt$. That is, given  a test of one kind,  we can  effectively determine a test of the other kind so that a real    fails the first test if and only if it fails the second test. \end{proposition}
 
 \begin{proof} Firstly, suppose we are given  an interval sequence test 
	\bc $(Q^m_r(k))_{r,k \in \NN}$, $(E^m_r)_{r\in\NN}$  ($m \in \NN$). \ec
	Let $G_m$ be the class in (\ref{eqn:StestFail}). Then $G_m$ is  $\SI 1 (\Halt)$ uniformly in $m$, and $\leb (G_m) $ is computable relative to $\Halt$ by (\ref{eqn:StestPart}).

  Secondly,  suppose we are given a Schnorr test  $(G_m) _{m \in \NN}$ relative to $\Halt$. Uniformly in $m$, using $\Halt$, we can compute $\leb (G_m)$ 
 for each $m \in \NN$. 
 Hence we  can for each $r, m  \in \NN$ determine $u_r\in \NN$ 
and, by possibly splitting into pieces some intervals 
%$[x]$ for which $[x] \subseteq$  
from $G_m$,
a finite sequence of rational intervals $P^m_r(i)$,   $u_r  < i  \leq u_{r+1}$,
   such that 
$\lambda (\bigcup_{u_r  < i  \leq u_{r+1}} P^m_r(i) )  \leq  2^{-(m+r)}$  and 
$G_{m} = \bigcup_r \bigcup_{u_r < i \leq u_{r+1}} P^m_r(i)$. 
By the Limit Lemma we have a computable sequence of intervals $P^m_r(i,t)$ and a computable sequence $u_r(t)$, $t \in \NN$,  
 such that for large enough $t$,  $u_r(t) = u_r$ and $P^m_r(i,t)= P^m_r(t)$ for $i\le u_r$. 
 From this we can build an interval sequence test as required: the uniformly c.e.\ finite sets $E^m_r$ correspond to the  intervals we want to remove because of the mind changes of the approximations $u_r(t)$ and $P^m_r(i,t)$ for $i \le u_r(t)$. 
  \end{proof}
  
In later work \cite{Demuth:88}, Demuth defined a fully relativized version of measure zero sets.  For a set of natural numbers $B$, a set $\+ S\subseteq[0,1]$ has $B$-measure zero if there is a $B$-Schnorr test $(G^B_m) _{m \in \NN}$ such that $\+ S\subseteq\bigcap_m G^B_m$.  In keeping with his extended constructivism, Demuth only applied the notion of a $B$-measure zero set for sets $B$ such that $B\leq_T\emptyset^{(n)}$ for some $n$. Nevertheless, the definition is stated in full generality.  In fact, Demuth defined the more general notion of \emph{$B$-measurability} for a given $B\subseteq\NN$, of which the notion of a $B$-measure zero set is a special case.

\subsection{Demuth's version of \ML-randomness}

Several other randomness notions arose in Demuth's study of the differentiability of Markov computable functions.  It  was natural for Demuth to consider a broader class of reals than just the computable reals, as computable reals do not suffice to study the points of differentiability of these functions.  For instance, Demuth proved that the derivative of a Markov computable function at a computable real need not be computable.  He also proved the existence of an absolutely continuous Markov computable function that is not pseudo-differentiable at any computable real (where pseudo-differentiability is defined below).  Demuth further showed that this function is only pseudo-differentiable at Martin-L\"of random reals, and, as we will discuss in the next subsection, at all of them.
     
In a 1975 paper \cite{Demuth:75a}, Demuth   introduced  a randomness notion equivalent to  \ML\ randomness.  At the time of the publication of \cite{Demuth:75a}, Demuth was not aware of \ML's earlier definition in~\cite{Martin-Lof:66} dating from 1966.  Demuth originally considered only Martin-L\"of random pseudo-numbers, which he called \emph{$\Pi_2$-numbers}.  As  a constructivist, Demuth found it more natural to   define the \emph{non}-\ML\ random pseudo-numbers first. He called  them \emph{$\Pi_1$-numbers}.

\begin{definition} A $\Delta^0_2$ real $x$ is a \emph{$\Pi_1$-number} 
if there is a computable sequence of rationals $(q_n)_{\sN{n}}$ with $x= \lim_{n\rightarrow\infty} q_n$  and 
a computable sequence of finite computable sets $(C_m)_{\sN{m}}$  such that  
$\lambda( \bigcup_{n \notin C_m} [q_n,q_{n+1}])  < 2^{-m} $.
\end{definition}

We provide a sketch of the proof that a $\Delta^0_2$ real $x \in [0,1]$ is a $\Pi_1$-number if and only $x$ is not  \ML\  random. 
For one implication, from a computable sequence of rationals $(q_n)_{\sN{n}}$ with $x= \lim_{n\rightarrow\infty} q_n$  and 
a computable sequence of finite computable sets $(C_m)_{\sN{m}}$  such that  
$\lambda( \bigcup_{n \notin C_m} [q_n,q_{n+1}])  < 2^{-m}$, 
we can construct a \ML-test $(B_m) _{m \in \NN}$ by setting
\[
B_m = \{y : \ex n,k [|y - q_n| < 2^{-m-1-k} \ \wedge \ \# \{j : j \leq n, j \in C_{m+1} \} = k ]\}. 
\]
It is not hard to verify that $\lambda(B_m)\leq 2^{-m}$ for every $m$ and that $x \in \bigcap_{m \in \NN} B_m$.

For the reverse implication, let $(U_m) _{m \in \NN}$ be a universal \ML-test, i.e. a Martin-L\"of test $(U_m) _{m \in \NN}$ such that for any $z\in[0,1]$, $z$ is Martin-L\"of random if and only if $z\notin\bigcap_{m\in\NN}U_m$.  

Recall that $A\subseteq\{0,1\}^*$ is prefix-free if for every $\sigma,\tau\in\{0,1\}^*$, if $\sigma\in A$ and $\tau$ properly extends $\sigma$, then $\tau\notin A$.  Now let $(V_{m})_{m \in \NN}$ be a prefix-free subset of $\{0,1\}^*$ for which $U_m = \bigcup_{\sigma \in V_{m}} [\sigma)$ for any $m\in\NN$, where $[\sigma)$ is the interval $[0.\sigma,0.\sigma+2^{-|\sigma|})$ as in the proof of Theorem \ref{thm-uniform-continuity}.

Suppose  $x \in [0,1]$  is a $\Delta^0_2$ non-\ML\ random real and $x= \lim_{n\rightarrow\infty} q_n$ for a computable sequence of rationals $(q_n)_{\sN{n}}$ from $[0,1]$. We let 
\[
C_m = \{n : \mathrm{Hop}_m(q_n, q_{n+1}) \},
\] 
where the condition $\mathrm{Hop}_m(q_n, q_{n+1}$) means that $q_n$ and $q_{n+1}$ belong to intervals represented
by two strings from $V_m$ that  are   not    contiguous.  That is, $q_n \in [\sigma)$ and $q_{n+1} \in [\tau)$ for some $\sigma, \tau \in V_m$ such that $0.\sigma + 2^{-|\sigma|} \neq 0.\tau$ and  $0.\tau + 2^{-|\tau|} \neq 0.\sigma$.
$C_m$ is clearly a computable set uniformly in $m$. Since there is a $\sigma \in V_m$ such that $x \in [\sigma)$,  it is easy to verify that each $C_m$ 
is finite. This concludes our sketch of the equivalence. 
%
%
%If $x\in [0,1]$  is a $\Delta^0_2$ non-\ML\ random real and $x= \lim_{n\rightarrow\infty} q_n$ for a computable sequence of rationals $(q_n)_{\sN{n}}$ from $[0,1]$,
%then we set
%\begin{align*}
%      & C_m=\left\{ n : (\exists\sigma,\tau\in V_m)\;0.\sigma + 2^{-|\sigma|} < 0.\tau\;\;\&\; \right. \\
%      & \hspace{1cm} \left. q_n \in [0.\sigma, 0.\sigma + 2^{- |\sigma|})\;\;\&\;\; q_{n+1} \in [0.\tau, 0.\tau + 2^{- |\tau|}) \right\}.
%\end{align*}
%$C_m$ is clearly a computable set for any $m$  and since there is a $\sigma \in V_m$ such that $x \in [0.\sigma, 0.\sigma + 2^{- |\sigma|})$,  it is easy to verify that each $C_m$ 
%is finite.

In addition to defining $\Pi_1$-numbers and their complement, the  $\Pi_2$-numbers,  Demuth constructed  a universal \ML-test \cite[Theorems 2 and 6]{Demuth:75a}, albeit in different terminology: he built a computable sequence of rational intervals $(\mathcal{K}^t_s)_{\sN{t,s}}$   for which 
	$\leb(\bigcup_s \mathcal{K}^t_s) <  2^{-t}$  for all $t$  and such that any $\Delta^0_2$ real $x$ is a $\Pi_1$-number if and only if  $x \in \bigcup_s \mathcal{K}^t_s$ for all $t$. Furthermore, he showed that the property of a $\Delta^0_2$ real $x$ to be a $\Pi_1$-number does not depend on the choice of a computable sequence $(q_n)_{\sN{n}}$ 
   with $x= \lim_{n\rightarrow\infty} q_n$ (see \cite[Corollary 1 of Theorem 5]{Demuth:75a}).

Demuth also studied an analogue of Solovay tests in \cite{Demuth:75a}.  As stated at the beginning of this section, a real is Solovay random if and only if it is Martin-L\"of random. Significantly, a restricted version of this result was also established by Demuth, who proved that a  $\Delta^0_2$ real $x$ is a $\Pi_2$-number if and only if it is Solovay random \cite[Corollary 2 of Theorem 5]{Demuth:75a}.  We should note that Demuth's proof is easily extendible to hold for all reals, not just the $\Delta^0_2$ reals.

In \cite{Brodhead.Downey.ea:12}, it is shown that the \ML\ random $\DII$-reals are precisely the finitely bounded random reals, which are defined in terms of Martin-L\"of tests $(G_m)_{m\in\NN}$ where each $G_m$ is a finite union of intervals.  In \cite{Demuth:75a}, Demuth anticipated this result by proving that his definition of $\Pi_2$-number is equivalent to one given in terms of finitely bounded tests.

Another topic that Demuth investigated was the extent to which $\Pi_1$-numbers are preserved under basic arithmetical operations.  His main result on the subject, given in 
%
%\cite{Demuth:75},   CORRECTED BY TONDA on June 16 to \cite{Demuth:75a}  
%
\cite{Demuth:75a}
 is that for every $\Delta^0_2$ real $\alpha$, there exist $\Pi_1$-numbers $\beta_1$ and $\beta_2$ such that $\alpha=\beta_1+\beta_2$.  Thus, since such an $\alpha$ can be a $\Pi_2$-number, the sum of two $\Pi_1$-numbers need not be a $\Pi_1$-number. 

Recall that a real $\alpha$ is left-c.e.\ if $\alpha$ is the limit of a computable, non-decreasing sequence of rational numbers. 
 Demuth further showed that the situation differs significantly if we consider pseudo-numbers that are left-c.e.: if $\beta_1$ and $\beta_2$ are left-c.e.\ $\Pi_1$-numbers, then $\beta_1+\beta_2$ is also a left-c.e.\ $\Pi_1$-number 
%
% TONDA: I have added the reference (June 16)
%
 (see \cite{Demuth:75a}). 
In other words if $\alpha$ is left-c.e.\ and Martin-L\"of random, and   $\alpha=\beta_1+\beta_2$ for left-c.e.\ reals $\beta_1$ and $\beta_2$, then at least one of $\beta_1$, $\beta_2$ is Martin-L\"of random.  This   is one of the earliest results in the theory of left-c.e.\ reals, a subject developed by Solovay in \cite{Solovay:75} that has been of much interest in recent years (see, for instance, \cite{Downey.Hirschfeldt.ea:02}  where Demuth's result is rediscovered, as well as Chapters 5 and 9 of \cite{Downey.Hirschfeldt:book}).

Beginning in 1978, Demuth was willing to countenance arithmetical reals. For instance, in \cite{Demuth.Kryl.ea:78}, Demuth, Kryl, and Ku\v cera prove that pseudo-numbers relative to $\emptyset^{(n)}$  correspond   to computable reals relative to $\emptyset^{(n+1)}$. In \cite{Demuth:82a} he refers to the arithmetical \emph{non}-\ML\ -random reals as $\+ A_1$ numbers and the arithmetical \ML\  random reals $\+ A_2$ numbers.  For instance, the definition of $\+ A_1$ can be found in  \cite[page 457]{Demuth:82a}. By then, Demuth knew of \ML's work: he defined $\+ A_1$ to be 
%
%
%$\bigcup_k [W_{\la g \ra(k)}]$, 
%
%
$\bigcap_k \Opcl {W_{g(k)}}$, 
where $g$ is a computable function determining a universal \ML\ test, and $\Opcl X$ is the set of arithmetical reals extending a string in $X$.
In the   English language papers such as \cite{Demuth:88}, the non-\ML\ random reals were called AP (for approximable in measure), 
and the \ML\ random reals were called NAP (for non-approximable in measure).

\subsection{Differentiability and randomness}  \label{ss:diff_rd}
Before discussing the application of Demuth's version of \ML-randomness to differentiability of Markov computable functions, we will review some definitions and provide some terminology.  For  a   function $f$,  the \emph{slope} at a pair $a,b$  of distinct reals in its domain~is 
   \[ S_f(a,b) = \frac{f(a)-f(b)}{a-b}.\] 
 Recall that if $z$ is in the domain of $f$ then  \begin{eqnarray*} \ol D f(z)  & = & \limsup_{h\ria 0} S_f(z, z+h) \\
   \ul D f(z) & = & \liminf_{h\ria 0} S_f(z, z+h)\end{eqnarray*} 
   Note that we allow the values $\pm \infty$. 
By the definition, a   function  $f$ is    differentiable at $z$ if $\ul D f(z) = \ol D f(z)$ and this value is finite.  We will denote  the derivative of $f$ at $z$ by $f'(z)$.

\label{ss:PseudoDer} If one wants to study the differentiability of   Markov computable  functions, one immediately runs into the problem that these functions are only defined on the computable reals. So one has to  introduce  upper and lower ``pseudo-derivatives'' at a real $z$, taking the limit of slopes close to $z$ where the function is defined.      This is precisely what Demuth did. 
   Consider  a function $g$ defined on $I_\QQ$,  the rationals in  $[0,1]$. For  $z \in [0,1]$ let
    
    \vsp

       $\widetilde  Dg(z) = \limsup_{h \to 0^+}  \{S_g(a,b)  \colon a, b \in I_\QQ   \lland  \, a\le z \le b \lland\, 0 <  b-a\le h\}$.
   
   \vsps
       
          $\utilde  Dg(z) = \liminf_{h \to 0^+}  \{S_g(a,b)  \colon a, b \in I_\QQ   \lland  \, a\le z \le b \lland\, 0 <  b-a\le h\}$.

   \begin{definition} {\rm We say that  a function $f$ with domain containing $I_\QQ$    is \emph{pseudo-differentiable at} $x$ if  $-\infty < \utilde  D f(x)  =   \widetilde D f(x) < \infty$, in which case the value $\utilde  D f(x)  =   \widetilde D f(x)$ will be denoted $f'(x)$.} \end{definition}

Since Markov computable functions are continuous on the computable reals,   it does not matter which    dense set of 
      computable reals one takes in the definition of these  upper and lower pseudo-derivatives.  For instance, one could take all computable reals, or only the dyadic rationals.  For a total continuous function~$g$, we have  $\utilde  Dg(z) = \ul Dg(z)$ and    $\widetilde  Dg(z) = \ol Dg(z)$.  The  last section of the extended  arXiv version  of~\cite{Brattka.Miller.ea:nd} contains more detail on pseudo-derivatives.

 Initially, Demuth studied the pseudo-differentiability of Markov computable functions at computable reals in the 1969 paper~\cite{Demuth:69b}.  For reasons mentioned above, in this limited setting, the resulting theory of pseudo-differentiability was not adequate. However, in a second 1975 paper~\cite{Demuth:75}, Demuth considered pseudo-differentiability at pseudo-numbers, which enabled him to prove a number of significant results
 on the pseudo-differentiability of Markov computable functions of bounded variation.
  The abstract of the paper,  translated literally, is as follows: 

\begin{quote} It is shown that every constructive function $f$  which cannot fail to be a function of weakly bounded variation is finitely   pseudo-differentiable on each $\Pi_2$-number.

 For almost every pseudo-number $\xi$ there is a pseudo-number which is a value of pseudo-derivative of the function $f$  on $\xi$, where the differentiation is almost uniform. \end{quote}

We rephrase Demuth's result in modern terminology.
\begin{theorem}[\cite{Demuth:75}]\label{thm:Demuth-BV}
Let $f$ be a Markov computable function of bounded variation. 
\begin{itemize}
\item[(i)] $f$ is pseudo-differentiable at any $\Delta^0_2$ {\ML} random real. 
\item[(ii)] Furthermore, there is a Schnorr test relative to $\emptyset'$ such that 
for any $\Delta^0_2$ real $\xi$ passing the test, $\xi$ is {\ML} random,  $f'(\xi)$ exists, and  $f'(\xi)$ is a  $\Delta^0_2$ real which can be computed uniformly in $\emptyset'$ and the representation of $\xi$ as a $\Delta^0_2$ real.
\end{itemize}
\end{theorem}

%**********
%Converted into  modern language, the first paragraph says that each  Markov computable function of bounded variation is (pseudo-)differentiable at each \ML{}  random real. 
%The first half of the second paragraph of the abstract of  \cite{Demuth:75} expresses  that 
%for almost every $\DII$ real $z$, the derivative $f'(z)$ is also $\DII$.  

%
%Lastly, in the second half of the second paragraph of the abstract, Demuth states that  ``the differentiation is almost uniform''. This means that there is a Schnorr test relative to $\emptyset'$ 
%%and a $\emptyset'$-computable function $g$ 
%such that for any $\Delta^0_2$ real $\xi$ passing this test, $f'(\xi)$ exists and can be uniformly computed relative to $\emptyset'$. 

To prove that a classical function $f$ of bounded variation is almost everywhere differentiable one usually 
expresses $f$ as a difference of two non-decreasing functions $f_1, f_2$. 
In the constructive setting, this approach no longer works, since a Markov computable function of bounded variation need not be expressible as a difference of two non-decreasing Markov computable functions, as proved by Ce\v\i tin and Zaslavski\v\i\ in \cite{Zaslavskii.Ceitin:1962}.

A  function $f$ is  called interval-c.e.\ \cite{Freer.Kjos.ea:14} if $f(0)=0$ and $f(y)-f(x)$ is a left-c.e.\ real, uniformly in rationals $x < y$.   If we wanted to  allow a little more leeway, any Markov computable function 
of bounded variation is expressible as $f_1 - f_2$, where $f_1, f_2$ are non-decreasing \emph{interval-c.e.\ functions}.  Unfortunately, functions of this more general type need not be differentiable at each Martin-L\"of random real as shown recently by Nies in
\cite[Theorem 7]{Nies:14}, so a different approach is needed.

For a c.e.\ set $C$ of closed rational intervals, we will write $\mathcal{H}(C)$ to denote that the intervals in $C$ are non-overlapping, i.e., they have at most endpoints in common, and  
that for any $k$ one can compute a stage $s$ in the enumeration of $C$ such that any interval enumerated into $C$ after stage $s$ has size less than $2^{-k}$.

The approach Demuth took to prove part (i) of Theorem \ref{thm:Demuth-BV} is roughly as follows.  First, for a given Markov computable function $f$ and a c.e.\ set $C$ of rational intervals such that $\mathcal{H}(C)$ holds, Demuth defines the function $[f,C]$ to be the Markov computable function such that for each interval $I$ in $C$, $[f,C]$ is equal to $f$ on the endpoints of $I$, is linear on the interior of $I$, and is equal to $f$ otherwise.

Demuth next proves the following lemma (see \cite[Lemma 4]{Demuth:75}).

\begin{lemma} \label{truncation} {\rm 
Let $f$ be a Markov computable function and $w, z$ computable reals such  that $w < z$, both $w,z$ are not equal to $(f(b) - f(a))/(b-a)$ for any rationals $a,b \in [0,1]$, and $f(1) - f(0) < z$.
Then there is a c.e. set $C$ of rational intervals with the following properties:
\begin{itemize}
\item[(i)] the condition $\mathcal{H}(C)$ holds;
\item[(ii)] $ w(b-a) <  f(b) - f(a) $ for any interval $[a,b]$ from $C$; and
\item[(iii)] $[f,C](y) - [f,C](x) < z(y-x)$ for any computable reals $x,y$ from $[0,1]$.
\end{itemize}
An analogous statement holds if replace each occurrence of $``<"$  with $``>"$ (including ``$w > z$" instead of ``$w < z$").  }
\end{lemma}

It follows that the function $g(x)=[f,C](x) -z\cdot x$ is strictly monotone on $[0,1]$ and $[f,C]$ is uniformly continuous of bounded variation.

In particular, if $f$ is a Markov computable function of bounded variation, then for any $k$, we can compute sufficiently large $w$ and $z$ with $w<z$ such that the measure of the intervals from the set $C$ guaranteed to exist by Lemma \ref{truncation} is less than $2^{-k}$ (and similarly for $z<w$).

Further, if we  apply Lemma \ref{truncation} twice, first to some appropriately chosen $w<z$ and then to some appropriately chosen $w'>z'$, we can truncate a Markov computable function of bounded variation $f$ into a Markov computable Lipschitz  function $[f,C]$ for some c.e.\ set $C$ of rational intervals with $\mathcal H(C)$.  Combining this with the statement in the previous paragraph, we can find a c.e.\ set $C_k$ of rational intervals effectively in $k$ such that $\mathcal{H}(C_k)$ holds and the measure of the intervals in $C_k$ is less than $2^{-k}$.

Next, by using a series of complicated approximations of Markov computable Lipschitz functions by Markov computable polygonal functions, Demuth proves that every Markov computable Lipschitz function is differentiable at every Martin-L\"of random real \cite[Theorems 1 and 2]{Demuth:75}.

Demuth then proves that if $f$ is a Markov computable function of bounded variation and $C$ is a c.e.\ set of rational intervals such that $\mathcal{H}(C)$ holds and such that the function $[f,C]$ is Lipschitz, the function $f-[f,C]$ is differentiable at any Martin-L\"of random real outside of any interval in $C$.  But since $f=[f,C]$ outside of any interval in $C$, it follows that $f$ is differentiable at any Martin-L\"of random real outside of any interval in of $C$.  

Lastly, using the fact that we can control the measure of the intervals in the above set $C$, we produce a uniformly c.e.\ collection $(C_k)_{k\in\NN}$ of rational intervals such that for every $k$, (a) $\mathcal{H}(C_k)$ holds, (b) the function $[f,C_k]$ is Lipschitz, and (c) the measure of the intervals in $C_k$ less than $2^{-k}$, where condition (c) implies that $(C_k)_{k\in\NN}$ defines a Martin-L\"of test. Thus for every Martin-L\"of random real $x$, there is some $k$ such that $x$ not in any interval in $C_k$, and since $[f,C_k]$ is differentiable at $x$, it follows that $f$ is differentiable at $x$ as well.  The concludes the proof of (i).

% to approximate a given Markov computable function of bounded variation $f$ by a computable sequence of polygonal functions $(g_n)_{n\in\NN}$ such that for every $n$ the variation of $f - g_n$ is less than $2^{-n}$ and showed that this yields the desirable result.\\

A more general version of Theorem \ref{thm:Demuth-BV} (i), which holds for \emph{all} Martin-L\"of random reals, has been recently reproved in \cite[Thm.\ 6.7]{Brattka.Miller.ea:nd} in an indirect way. It relies  on  a similar result for computable randomness: each Markov computable non-decreasing function is differentiable at each computably random real. The latter result is  in the same paper~\cite[Theorem 4.1]{Brattka.Miller.ea:nd}, taking into account the    extension of the theorem  in the last section of the arXiv version.

According to part (ii) of Theorem \ref{thm:Demuth-BV}, there is a single Schnorr test relative to $\emptyset'$ such that 
for any $\Delta^0_2$ real $\xi$ passing the test, $\xi$ is {\ML} random, $f'(\xi)$ exists and is $\Delta^0_2$, and $f'(\xi)$ can be uniformly computed from $\emptyset'$ and the representation of $\xi$ as a $\Delta^0_2$ real.

To prove this, Demuth carried out detailed calculations to produce the desired Schnorr test relative to $\emptyset'$.  We can reprove Theorem \ref{thm:Demuth-BV} (ii) as follows.  First, since $f$ is Markov computable, it
is easy to verify that  
\[
f'(z) \le_T z',
\]
 namely, the value of the pseudo-derivative of $f$ at $z$ is computable in the Turing  jump of $z$ whenever this pseudo-derivative exists.  Thus $f'(z)$ is $\DII$ whenever $z$ is low. Moreover, by \cite[Remark 10, part 3b]{Demuth:88}, or \cite[Theorem 3.6.26]{Nies:book},  there is a single Schnorr test relative to $\Halt$ (in fact, a Demuth test as defined in Definition~\ref{df:DemDef} below) such that each real $z$ passing it is generalized low, i.e., $z' \le_T z \oplus \Halt$.  Thus, the only reals $z$ for which $f'(z)$ is not $\DII$ are captured by this Schnorr test relative to $\emptyset'$.

Moreover, there is a fixed effective procedure for computing $z'$ from $z \oplus \ES'$ for any $z$ 
passing this Schnorr test (see the proof of Theorem 3.6.26 in \cite{Nies:book}). This yields the desired uniform computability of $f'(z)$  
from $\ES'$ and the representation of $z$ as a $\Delta^0_2$ real for any $\Delta^0_2$ real $z$ passing the test.\\

\subsection{The Denjoy alternative}\label{subsec-denjoy-alternative}

Demuth also closely studied the Denjoy alternative for Markov computable functions. One simple version of the  Denjoy alternative for a function $f$ defined on the unit interval says   that
	\begin{equation} \label{eqn:DA} \text{either $f'(z)$ exists, or $\ol Df(z) = \infty $ and $\ul Df(z) = -\infty$}. \end{equation}
	The full  result  is given in terms of  left and right upper and lower Dini derivatives, but we consider only the more compact version here.  
	
	 It is a consequence of the  classical  Denjoy (1907), Young (1912), and  Saks (1937)  Theorem    that for \emph{any} function  defined on the unit interval, the  Denjoy alternative holds at almost every $z$. Denjoy himself obtained the Denjoy alternative for continuous functions, Young for measurable functions,  and Saks for all functions. 
 For a proof see for instance Bogachev~\cite[p.\ 371]{Bogachev.vol1:07}.   
   
   %One application of this  result is to show that $f'$ is   Borel (as a partial function) for any function~$f$. A paper by Alberti et al.\ \cite{Alberti.Csornyei.ea:00}    revisits the  Denjoy alternative. They provide a  version that is in a sense optimal. 

Here we formulate the Denjoy alternative  in terms of pseudo-derivatives.

\begin{definition} {\rm Suppose the domain of a partial function $f$ contains $I_\QQ$.  We say that the \emph{Denjoy alternative}  holds for $f$ at $z$  if 
	\begin{equation} \label{eqn:pseudoDA} \text{either $\widetilde Df(z) = \utilde Df(z) < \infty$, or $\widetilde Df(z) = \infty $ and $\utilde Df(z) = -\infty$}. \end{equation}
 } \end{definition}
 
\begin{equation}
 \utilde{D}f(z)
\end{equation}
\noindent This is equivalent to  (\ref{eqn:DA}) if the function is total and continuous. 

%Cater~\cite{Cater:86} has given an alternative  proof of a   stronger fact: the reals $z$ where the  right lower derivative $D_+(z)$ is infinite     form a null set. 

 For any  function $g\colon \, [0,1] \to \RR$, the reals $z$ such that  $\ul Dg(z) = \infty$  form a null set. This  well-known fact from classical analysis is usually proved via covering theorems, such as 
Vitali's or Sierpinski's.  (Cater~\cite{Cater:86} has given an alternative proof of a stronger fact: the reals $z$ where the right lower derivative $\ul D_+g(z)$ is infinite form a null set.)

Demuth was interested in determining which type of  null class is needed to make an analog of this  classic fact  hold for Markov computable functions (see Definition~\ref{def:Markov_computable}).    
The following notion can be found in~\cite{Demuth:80} (although a variant was given in the earlier \cite{Demuth:78}). As usual, for functions not defined everywhere we have to work with pseudo-derivatives as defined in Subsection~\ref{ss:diff_rd}.

\begin{definition} \label{df:DenjoyRandom}  {\rm  A  real $z \in [0,1]$ is called  \emph{Denjoy random} (or a \emph{Denjoy set}) if  for no Markov computable function $g$ do we have $\utilde D g(z) = \infty$. } \end{definition}

We should emphasize here that Demuth only used the term ``Denjoy set" in the preprint and final version of his paper 
 ``Remarks on Denjoy sets''~\cite{Demuth:preprint88}. The preprint was based on a talk   Demuth gave  at the Logic Colloquium 1988 in Padova, Italy (close to the end of communist era  in 1989, as it became  easier to travel to the  ``West").   He   later turned the preprint survey  into the paper~\cite{Demuth:90} with the same title, but it  contains only part of the  preprint survey.

%  The  \href{http://dl.dropbox.com/u/370127/DemuthPapers/Demuth80DenjoySets.pdf}{paper}~\cite{Demuth:80} is entitled ``The constructive analogue of a theorem by Garg on derived numbers''.   Garg's Theorem, a variant  of the Denjoy-Young-Saks theorem discussed in Subsection~\ref{ss:Denjoy},   has  the   somewhat obscure reference~\cite{Garg:61}. 

As reported in the preprint survey~\cite[p. 6]{Demuth:preprint88}, in \cite{Demuth:78} it is shown that 
if  $z \in [0,1]$ is  Denjoy random, then for every $\ES$-uniformly continuous Markov computable $f \colon \, [0,1] \to \mathbb R$ the Denjoy alternative  (\ref{eqn:DA}) holds at~$z$.  
Combining this with  the results in \cite{Brattka.Miller.ea:nd} we can now determine precisely what Denjoy randomness is, and also obtain  a pleasing new characterization of computable randomness of reals through differentiability of standard computable functions. 

\begin{theorem}[\cite{Bienvenu.Hoelzl.ea:12a}]  \label{thm:DenjoyCR}  The following are equivalent for a real $z \in [0,1].$
\bi 

\item[(i)] $z$ is Denjoy random. 

\item[(ii)]  $z$ is computably random  

\item[(iii)]  for every standard computable $f \colon \, [0,1] \to \mathbb R$ the Denjoy alternative  (\ref{eqn:DA}) holds at~$z$.    \ei \end{theorem}

\n \begin{proof} (i)$\to$(iii) is Demuth's result (see \cite[Theorem 1]{Demuth:80} and \cite[Theorem  3]{Demuth:78}). 
For  (iii)$\to$(ii), let $f$ be a non-decreasing standard computable function. Then $f$ satisfies the Denjoy alternative at $z$. Since $\ul Df(z) \ge 0$, this means that $f'(z)$ exists. 
This implies that  $z$ is computably random by~\cite[Thm.\  4.1]{Brattka.Miller.ea:nd}.

The implication (ii)$\to$(i) is proved by contraposition: if $g$ is Markov computable and $\utilde Dg(z) = \infty$ then one builds a computable martingale that witnesses that $z$ is not computably random. See \cite[Thm.\ 15]{Bienvenu.Hoelzl.ea:12}  or~\cite{Bienvenu.Hoelzl.ea:12a} for the details of the proof. \end{proof}

 \begin{remark} \label{rem:Weaker} {\rm For the contraposition of the    implication (ii)$\to$(i), it suffices to use the weaker hypothesis on $g$ that $g(q)$ is a computable real uniformly in a rational $q \in I_\QQ$.  } \end{remark}

% \begin{remark} \label{rem:Weaker} {\rm For the contraposition of the    implication (ii)$\to$(i), actually the weaker hypothesis on $g$ suffices  that $g(q)$ is a computable real uniformly in a rational $q \in I_\QQ$.  } \end{remark}

 %Thus, in Definition~\ref{df:DenjoyRandom}  we can replace Markov computability of $g$ by this weaker hypothesis without changing the class of reals defined.

We do not fully understand how Demuth obtained (i)$\to$(iii) of the theorem;  a proof of this using  classical language would be   useful.  We can, however, obtain a direct proof of the contraposition of (i)$\to$(ii) that uses techniques from modern algorithmic randomness (which can be found in \cite[Thm.\ 3.6]{Brattka.Miller.ea:nd}): if $z$ is not computably random then a martingale~$M$ with the so-called ``savings property'' succeeds on (the binary expansion of) a real $z$.  Recall that $M$ has the savings property if $M(\tau)\geq M(\sigma)-2$ for every pair of strings $\tau\succeq\sigma$.  The authors now  build a standard computable function $g$ such that $\ul Dg(z) = \utilde Dg(z) = \infty$. 

Together with   Remark~\ref{rem:Weaker} we obtain: 

\begin{cor}  The following are equivalent for a real $z$:
\bi \item[(i)]  For no   function $g$ such that $g(q)$ is uniformly computable  for $q \in I_\QQ$ do we have $\utilde D g(z) = \infty$.
    \item[(ii)] $z$ is Denjoy random, i.e., for no Markov computable function $g$ do we have $\utilde D g(z) = \infty$.
    \item[(iii)] For no  standard computable function $g$ do we have $\ul D g(z) = \infty$. \ei 
    \end{cor}
    This implies  that the particular choice of Markov computable functions in Definition  \ref{df:DenjoyRandom} is irrelevant. Similar equivalences stating that the exact level of effectivity of functions does not matter have been obtained in the article~\cite{Brattka.Miller.ea:nd}. For instance, the version of Theorem \ref{thm:Demuth-BV} (i) from \cite{Brattka.Miller.ea:nd} holds for any functions of bounded variation with any of the three particular effectiveness properties above: standard computable, Markov computable, and uniformly computable on the rationals. For non-decreasing \emph{continuous} functions, the three effectiveness properties coincide as observed in~\cite[Prop.\ 2.2]{Brattka.Miller.ea:nd}.
    
Because of Theorem~\ref{thm:DenjoyCR} one could assert  that  Demuth  studied  computable randomness indirectly via his Denjoy sets. Presumably he didn't know the  notion of computable randomness, which was  independently introduced by  Schnorr in~\cite{Schnorr:71} (see also \cite[Ch.\ 7]{Nies:book}  or \cite[Section 7.1]{Downey.Hirschfeldt:book}).    Demuth also proved in \cite[Thm.\ 2]{Demuth:88} that every Denjoy set that is AP (i.e., non-Martin-L\"of random) must be high. The analogous result for computable randomness was later obtained in~\cite{Nies.Stephan.ea:05}.   There, the authors also show a kind of   converse:  each high degree contains a computably random set that is  not \ML\ random. This fact was apparently not known to Demuth (although he did prove a closely related result, as we will see in $\S$\ref{subsec-semigenericity} in our discussion of semigenericity).

As mentioned above, Demuth knew that Denjoy randomness of a real $z$ implies the Denjoy alternative  at $z$ for all standard computable functions. It was thus natural for Demuth to ask the following question:
\begin{quote}
 \n \emph{How much  randomness   for a  real  $z$ is needed to ensure the Denjoy alternative at  $z$ for all Markov computable functions?} 
\end{quote}
Demuth  showed the following (see the preprint survey, \cite[p.\ 7, Theorem 5, item~4]{Demuth:preprint88}, which refers to \cite{Demuth:76}).

\begin{theorem}  There is a Markov computable function $f$ such that the Denjoy alternative fails at some Martin-L\"of random real~$z$.   Moreover,  $f$ is extendable to a continuous function on $[0,1]$.
\end{theorem} 

This theorem has been reproved by Bienvenu, H\"olzl, Miller and Nies~\cite{Bienvenu.Hoelzl.ea:12,Bienvenu.Hoelzl.ea:12a}. In their proof, $z$ can be taken to be  the least element of  an arbitrary effectively  closed set of reals containing only Martin-L\"of random reals. In particular,  one can make $z$ left-c.e.

It was now clear to   Demuth that a randomness notion stronger than \ML's was needed. Such a notion  was introduced in  the  paper   ``Some classes of arithmetical reals'' \cite[p. 458]{Demuth:82a}. The definition is  reproduced in the preprint survey~\cite[p. 4]{Demuth:preprint88}.  In modern language the definitions are as follows.

\begin{definition} \label{df:DemDef} A \emph{Demuth test}  is a sequence of c.e.\ open sets $(S_m)\sN{m}$ such that $\fao m \leb (S_m) \le \tp{-m}$, and there is a function $f \colon \, \NN \to \NN$ with  $f\lwtt \Halt$ such that  $S_m = \Opcl{W_{f(m)}}$.

  A set~$Z$ \emph{passes} the test if $  Z\not \in   S_m$ for almost every~$m$.
    We say that~$Z$ is  \emph{Demuth random}  if~$Z$ passes  each Demuth test. \end{definition}

  Recall that  $f\lwtt \Halt$ if and only if $f$ is $\omega$-c.e., namely, $f(x) = \lim_t g(x,t)$  for some computable function~$g$ such that the number of stages~$t$ with  $g(x,t)\neq g(x,t-1)$ is computably  bounded in~$x$.  Hence the idea is that we can change the $m$-th component $S_m$ a computably bounded number of times.

\begin{figure}[hbt] 
	\bc \scalebox{.98}{\includegraphics{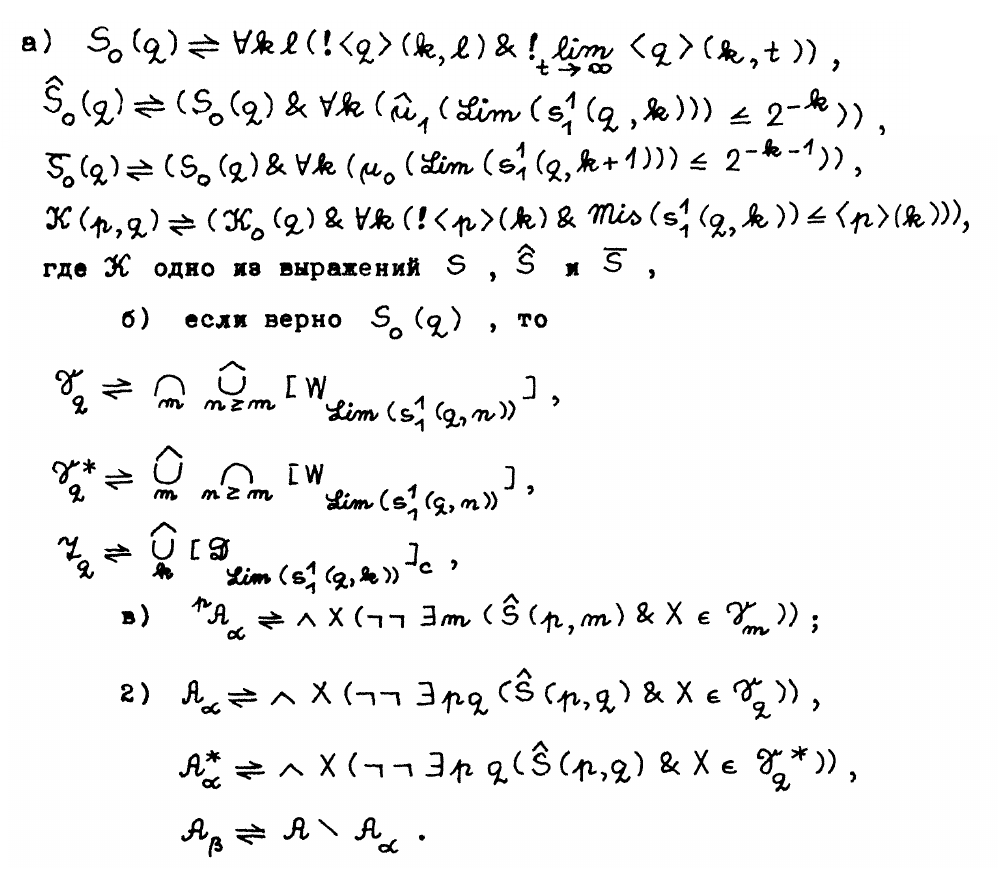}} \ec
	\vspace{-.5cm}
	\caption{\cite[p. 458]{Demuth:82a}: $\+ A_\beta$ is the definition of Demuth randomness}
	  \label{Fig:DemuthDef}
\end{figure}

Fig.\ \ref{Fig:DemuthDef} shows   the definition of Demuth randomness as it appears in the 1982  paper  \cite[p.\ 458]{Demuth:82a}.  For a given index $q$ of a binary computable function $\phi_q(k,x)$, Demuth defines the set
   
\bc 
$\Upsilon_q=\{Z:(\fa m)(\ex k \ge m)$   $Z\in\Opcl{W_{\lim (s^1_1(q,k))}}\}$,
\ec 
provided that $\lim (s^1_1(q,k))$ (which simply means $\lim_x \phi_q(k,x)$,  the final version $r$ of the test) exists.  A further condition $\mathcal K(p,q)$, involving an index $p$ for a computable unary function, yields the bound $\phi_p(k)$ on the number of changes.  The bound $\tp{-k}$ on measures of the $k$-th component can   be found in part a) of Fig.\ \ref{Fig:DemuthDef}.  The notation \emph{Mis}$(s^1_1(q,k))$ in Fig.\ \ref{Fig:DemuthDef} refers to the number of ``mistakes'', i.e.\ changes, and Demuth requires it be bounded by $\la p \ra(k)$, meaning $\phi_p(k)$.

If we apply the usual passing  condition for tests, we obtain the following notion which only occurs in \cite[p. 458]{Demuth:82a}.

\begin{definition}\label{def:wdr} We say that a set  $Z\sub \NN$ is \emph{weakly Demuth random} if for each Demuth test $(S_m)\sN{m}$  there is an $m$ such that $Z \not \in S_m$.
\end{definition}
%
%
%In the context of the Definition \ref{df:DemDef}, if we also have $S_m \supseteq S_{m+1}$, for each $m$, we say that $(S_m)\sN{m}$  is a monotonic Demuth test. Thus, a set $Z$ is weakly Demuth random if and only if $Z$ passes all monotonic Demuth tests.

In \cite {Demuth:82a} weak Demuth randomness is defined in terms of a set $\Upsilon_q^*$, where the quantifiers are switched compared to the definition of  $\Upsilon_q$: 
\bc  
$\Upsilon^*_q=\{Z:(\ex m)  (\fa k \ge m)$   $Z\in\Opcl{W_{\lim (s^1_1(q,k))}}\},$ 
\ec
again provided that $\lim (s^1_1(q,k))$ exists.

Note, however, that there is a slight difference between the definition of weak Demuth randomness as given in \cite{Demuth:82a} and that given by Definition \ref{def:wdr} above.  
If we set $S_k=\Opcl{W_{\lim (s^1_1(q,k))}}$, then $Z\notin\Upsilon_q^*$ means $Z\notin S_m$ for infinitely many $m$.   However, the two definitions are equivalent, since for a given Demuth test  $(S_m)\sN{m}$, for each $m\in\NN$, $(S_k)_{k\geq m}$ also yields a Demuth test.

The class of arithmetical non-Demuth randoms is denoted $\+ A_\alpha$, and the class of arithmetical non-weakly Demuth randoms is denoted $\+ A_\alpha^*$.  The complement of $\+ A_\alpha$ within the arithmetical reals is  denoted  $\+ A_\beta$  
and, similarly, the complement of $\+ A_\alpha^*$   within the arithmetical reals is  called  $\+ A_\beta^*$.
Later on, in the preprint survey, Demuth   used the terms WAP sets (weakly approximable in measure) for the non-Demuth randoms, and NWAP for the Demuth randoms 
and  the terms WAP$^*$ sets and NWAP$^*$ sets for the non-weakly Demuth randoms and the weakly Demuth randoms, respectively.

In the preprint survey~\cite[p. 7, Thm 5, item 5)]{Demuth:preprint88}, 	Demuth  states that Demuth randomness is sufficient to guarantee that the Denjoy alternative for Markov computable functions holds (referring to \cite[Theorem 2]{Demuth:83}).

	\begin{theorem} \label{thm:DA_for_Markov_computable}  Let $z$ be a Demuth random real. Then  the Denjoy alternative holds at $z$ for every Markov computable function. \end{theorem}
	
	To derive this result, Demuth constructs a single Demuth test $(\+ S_n)_{n\in\NN}$ containing all non-Martin-L\"of random reals such that for any Markov computable function and any real $x$ one of the following holds:\begin{itemize}
  \item[(1)] $\widetilde Df(x) = +\infty $ and $\utilde Df(x) = -\infty$;
	\item[(2)] either $\utilde Df(x) > -\infty$  or  $\widetilde Df(x) < +\infty $, and one of the following holds:
	  \begin{itemize}
      \item[(i)] $\widetilde Df(x) = \utilde Df(x)$, $\lim_{r\rightarrow x} f(r)=y$ exists, $y\notin \+S_n$ for almost every $n$, and $\widetilde Df(x) \ne 0$;
      \item[(ii)] $\lim_{r\rightarrow x} f(r)=y$ exists but $y\in\+S_n$ for infinitely many $n$, i.e., $y$ does not pass the test $(\+ S_n)_{n\in\NN}$; 
      \item[(iii)] $\lim_{r\rightarrow x} f(r)$ does not exist.
    \end{itemize}

\end{itemize}

\n   If $x$ is Demuth random and $f$ is a Markov computable function it is possible to show that
\begin{itemize}
  \item condition (2)(iii) cannot hold for $x$. More precisely, as claimed in \cite[Remark 7]{Demuth:83}, condition (2)(iii) implies that $x$ is either a left-c.e.\ or a right-c.e.\ real (where a real is right-c.e.\ if it is the limit of a computable non-increasing sequence of rationals), which cannot be Demuth random;
	\item condition (2)(ii)	reduces to the situation where $f$ is differentiable at $x$ with the value $f'(x)$ equal to $0$;
	\item condition (2)(i) reduces to the situation where the value  $\widetilde Df(x) = \utilde Df(x)$ is finite and $f'(x) \neq 0$.
\end{itemize}
Thus, the Denjoy alternative for $f$ holds at any Demuth random real $x$.
%	
%
%To derive this result, Demuth constructs a single Demuth test $(\+ S_n)_{n\in\NN^+}$ containing all non-Martin-L\"of random reals such that for any Markov computable function and any real $x$ one of the following holds:\begin{itemize}
%  \item[(i)] $\widetilde Df(x) = +\infty $ and $\utilde Df(x) = -\infty$;
%  \item[(ii)] $\widetilde Df(x) = \utilde Df(x) < \infty$, $\lim_{r\rightarrow x} f(r)=y$ exists, $y\notin \+S_n$ for almost every $n$, and $\widetilde Df(x) \ne 0$;
% \item[(iii)] $\lim_{r\rightarrow x} f(r)=y$ exists but $y\in\+S_n$ for infinitely many $n$, i.e., $y$ does not pass the test $(\+ S_n)_{n\in\NN^+}$; or
%\item[(iv)] $\lim_{r\rightarrow x} f(r)$ does not exist.
%\end{itemize}

%$\Theta(x)$, a technical condition that essentially says that if $\lim_{r\rightarrow x} f(r)=y$ exists, then $y\in \+ S_n$ for infinitely many $n$.

%\n If $x$ is Demuth random it is possible to show that condition (iv) cannot hold for $x$ and that condition (iii) reduces to the situation where $f$ is differentiable at $x$ with the value $f'(x)$ equal to $0$.  Thus, the Denjoy alternative for $f$ holds at $x$.

%
%This result  is actually hard to pin down in \cite{Demuth:83}. Theorem 2 on page 399 comes close, but has some extra conditions not present in the original Denjoy alternative.  

\begin{remark} \label{rem:FrankNg} {\rm Franklin and  Ng~\cite{Franklin.Ng:10}  introduced difference randomness, a concept     much weaker than even weak Demuth randomness,  but still stronger than   \ML\ randomness.  Bienvenu, H\"olzl, Miller and Nies \cite[Thm.\ 1]{Bienvenu.Hoelzl.ea:12} have shown that difference randomness is sufficient as a hypothesis on the real $z$ in Theorem~\ref{thm:DA_for_Markov_computable}.  No converse holds. They also show  that the ``randomness notion'' to make the Denjoy alternative hold for each Markov computable function is incomparable with Martin-L\"of randomness! } \end{remark}
% See  the April 2011 entry of the   \href{http://dl.dropbox.com/u/370127/Blog/Blog2010.pdf}{Logic Blog} \cite{LogicBlog:11}.} 

\section{Further results on Demuth  randomness}\label{sec-dem-rand}

The notions of Demuth and weak Demuth randomness have proven to be very fruitful, being studied in a number of recent papers.  However, due to the relative inaccessibility of Demuth's work, many researchers in the field have been unaware of just how much Demuth proved about these notions.  In this section, we review some of Demuth's results on his notions of randomness.

\subsection{Computability-theoretic properties of Demuth randomness}

In the mid-1970s, the mathematics department at Charles  University held a seminar on computability theory based on Rogers' book \cite{Rogers:67}, which had been translated into Russian in 1972.  As a result of this seminar,  Demuth became more interested in computability theory and  the computational complexity of random reals.  

In particular, Demuth thoroughly studied the relationship between Demuth randomness and the Turing degrees.  For instance, in \cite{Demuth:88} he proved the following, which was already implicit in \cite[Theorem 6]{Demuth:82}.

\begin{proposition}
\begin{itemize}
\item[(i)] Every Demuth random real is generalized low, i.e.,  $z'\le_T z\oplus\ES'$.  
\item[(ii)] There is a single Demuth test $(S_m)_{m\in\NN}$ such that for every $z$ for which $z\in S_m$ for at most finitely many $m$, $z$ is generalized low.
\end{itemize}
\end{proposition}

Demuth actually proved a stronger result.  Recall that a truth-table reduction ($\mathrm{tt}$-reduction for short) is a Turing reduction given in terms of a computable sequence of truth-tables that determine the outputs of the reduction.  Equivalently, a $\mathrm{tt}$-reduction is a Turing reduction $\Phi$ such that $\Phi^X$ is total for all oracles $X$  (Nerode \cite{Nerode:57a}).  A $\ES'$-$\mathrm{tt}$-reduction is thus a reduction given in terms of a $\ES'$-computable sequence of truth-tables.  Demuth proved that for any Demuth random $z$, $z'$ is $\ES'$-$\mathrm{tt}$-reducible to $z$. Note that Demuth's result does not imply that $z'\equiv_{\mathrm{tt}}z\oplus\ES'$, since this latter statement implies that the use of $\ES'$ in the reduction is bounded by a computable function, which need not be the case for a $\ES'$-$\mathrm{tt}$-reduction.

Demuth also proved results about the growth rate of functions computable from Demuth random reals.  First, he showed that every Demuth random real has hyperimmune degree (i.e.\ that every Demuth random computes a function not dominated by any computable function).  In contrast, he also proved the following.

\begin{theorem}[Demuth \cite{Demuth:88a}]
There is a $\emptyset'$-computable function $g$ such that for every Demuth random $z$ and every  $z$-partial computable function $f$, $f(n)\leq g(n)$ for almost every $n$.
\end{theorem}
In modern terminology, this result implies that $\emptyset'$ is \emph{uniformly almost everywhere dominating}, a result  established earlier by Kurtz in \cite{Kurtz:81}.  What Kurtz showed is that there is a measure one set of reals $\+ S$ such that every total function computable from a member of $\+ S$ is dominated by a fixed $\ES'$-computable function.  Demuth was unaware of this result, but improved it in two ways, (1) by showing that $\+S$ includes every Demuth random real, and (2) by showing the function $g$ dominates every \emph{partial} function computable from every Demuth random.

  Demuth proved a further result of which   a variant of which was only recently rediscovered.
\begin{theorem}[Demuth \cite{Demuth:88}]\label{thm-DemR-MLR}
 Let $y$ be  Demuth random and $x$   Martin-L\"of random. If  $x\leq_Ty$ then  $x$ is Demuth random.
 \end{theorem}
 
Miller and Yu \cite{Miller.Yu:ta1} proved that for every 2-random $y$ (i.e.\ $y$ is Martin-L\"of random relative to $\ES'$) and Martin-L\"of random $x$, $x\leq_Ty$ implies that $x$ is 2-random (see also \cite[Theorem 8.5.3]{Downey.Hirschfeldt:book} or \cite[Corollary 3.6.20]{Nies:book}). This follows from their more general result that for any $z$, every Martin-L\"of random Turing below a $z$-Martin-L\"of random is also $z$-Martin-L\"of random.

Demuth's proof is very similar to the proof of the result of Miller and Yu given in \cite{Miller.Yu:ta1}. For a Turing functional $\Phi$ and  $n>0$, consider the open set  	\bc $S^A_{\Phi,n} =
      \Opcl{\{\sss\in \{0,1\}^*\colon \,  A\uhr n \preceq  \Phi^\sss\}}$. \ec
Miller and Yu proved that if $ A$ is \ML\ random then there is a constant $c$ such that $\fao n \leb (S^A_{\Phi,n}) \le \tp{-n+c}$ (see \cite[Lemma 10.3.7]{Downey.Hirschfeldt:book} or \cite[Theorem 5.1.14]{Nies:book}).  This method works for most test notions of randomness stronger than Martin-L\"of randomness.  An equivalent result (given in slightly different terminology) was obtained by 
Demuth and  \Kuc{} \cite[Theorem 18]{Demuth.Kucera:87}, which Demuth used in his proof of Theorem \ref{thm-DemR-MLR}.

Demuth also proved a version of the jump inversion theorem for Demuth random reals.

\begin{theorem}[Demuth Jump Inversion, \cite{Demuth:88}]\label{thm-demuth-jump-inversion}
For every $z\geq_T\ES'$, there is a Demuth random real $x$ such that $x'\equiv_Tz$.
\end{theorem}

An immediate corollary of Theorem \ref{thm-demuth-jump-inversion} is that there exists a $\Delta^0_2$ Demuth random real
\cite[Theorem 12]{Demuth:88}. For a direct proof of this corollary, see \cite[Theorem 7.6.3]{Downey.Hirschfeldt:book} or \cite[Theorem 3.6.25]{Nies:book}.

To prove Theorem \ref{thm-demuth-jump-inversion}, Demuth appealed to the following result.

\begin{theorem}[Demuth, \cite{Demuth:88}]\label{thm-pre-jump-inversion}
For $y,z\in[0,1]$ and any $\mathcal{E}\subseteq[0,1]$ of $y$-measure zero, there is $x\notin\mathcal{E}$ such that
$x\leq_T y\oplus z$ and
 $z\leq_T x\oplus y$.
\end{theorem}

The Demuth Jump Inversion theorem can be derived  from Theorem \ref{thm-pre-jump-inversion} as follows.  Let $z\geq_T\ES'$ be given, and let $y=\ES'$.
  Demuth proved that there is a single Schnorr test   $(G_m^{\ES'}) _{m \in \NN}$ relative to $\ES'$ that contains every non-Demuth random. We set $\mathcal{E}=\bigcap_{m \in \NN}G_m^{\ES'}$, so that $\mathcal{E}$ has $\ES'$-measure zero.  By Theorem \ref{thm-pre-jump-inversion}, there is some Demuth random $x$ such that $x\leq_T z\oplus\ES'\leq_T z$ and $z\leq_T x\oplus\ES'$.  It follows that $z\equiv_Tx\oplus\ES'$. Then,  since every Demuth random is generalized low, we have $z\equiv_T x'$.

\subsection{Weak Demuth randomness and density}

Another surprising result that Demuth proved involves the relationship between weak Demuth randomness and density in the sense of Lebesgue.  Only recently have researchers in the field recognized the significance of the relationship between randomness and Lebesgue density.  For instance, density considerations were used to solve a long-standing open problem known as the covering problem, originally due to F.\ Stephan, and  posed in print e.g.\ in~\cite{Miller.Nies:06}. This problem  asks whether every $K$-trivial  set is Turing below an incomplete ML-random set.  A survey of the affirmative solution  is given in \cite{Bienvenu.Day.ea:nd}.  Anticipating this connection between  randomness and density, already in  1982,  Demuth~\cite{Demuth:82}  proved a remarkable result. Recall that the lower density of a measurable set $\+ P$ at a real $z$ is 
\[
\rho(\+ P\mid z)=\liminf_{h\rightarrow 0}\{\leb(\mathcal{P} \cap I)/\leb(I):I \text{ is an open interval, }z\in I \;\&\; |I|< h \}.
\]

\begin{definition}
A real $z$ is a \emph{density-one point} if for every effectively closed class $\mathcal{P}$ containing $z$, $\rho(\+P\mid z)=1$.
\end{definition}

\begin{theorem}[Demuth, \cite{Demuth:82}]
Every weakly Demuth random is a density-one point.
\end{theorem}

Demuth actually proves a stronger result:  there is a single Demuth test $(S_m)\sN{m}$ such that every real for which $z\notin S_m$ for infinitely many $m$ is a density-one point.  A further strengthening was obtained by a group of researchers working at Oberwolfach at the beginning of 2012, who introduced a new notion of randomness that they called \emph{Oberwolfach randomness} (see \cite{Bienvenu.Greenberg.ea:nd}).  We give a definition equivalent to the original one in terms of left-c.e.\ bounded tests.

\begin{definition}
\begin{itemize}
\item[(i)] A \emph{left-c.e.\ bounded test} is an effective descending sequence $(U_m)_{m\in\NN}$ of open sets in $[0,1]$ together with computable increasing sequence of rationals ($\beta_m)_{m\in\NN}$ with limit $\beta$ such that $\lambda(U_m)\leq \beta-\beta_m$ for every $m$.
\item[(ii)] A real $z$ is \emph{Oberwolfach random} if and only if it passes every left-c.e.\ bounded test.
\end{itemize}
\end{definition}

By definition, $\beta$ is a left-c.e.\ real.  As the rate at which $(\beta_m)_{m\in\NN}$  converges to $\beta$ may not be bounded by a computable function, not every left-c.e.\ bounded test is a Martin-L\"of test.  However, since every Martin-L\"of test is a left-c.e.\ bounded test, it follows that every Oberwolfach random real is Martin-L\"of random.  Moreover, one can show that every weakly Demuth random real is Oberwolfach random. The implication is strict.

The Oberwolfach group proved the following, unaware of the fact that they were strengthening a  result of Demuth.

\begin{theorem}[Bienvenu, Greenberg, Ku\v cera, Nies, Turetsky \cite{Bienvenu.Greenberg.ea:nd}]
Every Oberwolfach random is a density-one point.
\end{theorem}

Determining the precise relationship between the following three classes is still open: 
\begin{itemize}
\item[(i)] the Martin-L\"of random reals that are not LR-hard, where a real $z$ is \emph{LR-hard} if every $z$-Martin-L\"of random real is   $\emptyset'$-Martin-L\"of random.
\item[(ii)] the Oberwolfach random reals, 
\item[(iii)] the collection of Martin-L\"of random density-one points.
\end{itemize}
 The known implications for Martin-L\"of random $z$ are as follows:

\begin{center}
$z$ is not LR-hard $\rightarrow$ $z$ is Oberwolfach random $\rightarrow$ $z$ is a density-one point
\end{center}

\subsection{Demuth randomness and lowness notions}

As discussed at the end of $\S$\ref{subsec-denjoy-alternative},  the Demuth randomness of a real  is much too strong for its original purpose, namely, ensuring that the Denjoy alternative holds at this  real for all Markov computable 
functions.  However, since Demuth randomness is stronger than \ML\ randomness but still compatible with being $\DII$,  it interacts well with certain computability-theoretic notions.  In particular, Demuth randomness  has recently turned out to be  very useful for the study of lowness notions. 

A lowness notion is given by a collection of sequences that are in some sense computationally weak.  Many lowness notions take the following form:  For a relativizable collection $\+ S\subseteq 2^{\NN}$, we say that $A$ is \emph{low for $\+S$} if $\+S\subseteq\+S^A$.  For instance, a sequence $A$ such that every Demuth random sequence is Demuth random relative to $A$ is \emph{low for Demuth randomness}.

Another lowness notion is that of being a \emph{base for randomness}.   For a randomness notion $\mathcal{R}$, $A$ is a \emph{base for $\mathcal{R}$-randomness} if $A\leq_T Z$ for some $Z$ that is $\mathcal{R}$-random relative to $A$.  If we let $\mathcal{R}$ be Demuth randomness, this yields the definition of being a base for Demuth randomness.

One additional lowness notion that has received much attention recently is known as \emph{strong jump traceability}.  Recall that a computable order $h$ is a non-decreasing, unbounded computable function such that $h(0)>0$.  If we let $J^A(n)$ denote $\Phi^A_n(n)$, then $A\in 2^\NN$ is \emph{$h$-jump traceable} for a computable order $h$ if there is a uniformly c.e.\ collection of sets $(T_e)_{e\in\NN}$ such that $|T_n|\leq h(n)$ and $J^A(n)\downarrow$ implies that $J^A(n)\in T_n$ for all $n$ (Nies, \cite{Muenster}).  Moreover, $A$ is \emph{strongly jump traceable} if it is $h$-jump traceable for all computable orders $h$ (Figueira, Nies, Stephan, \cite{Figueira.ea:08}).  The notion of tracing is due to Zambella \cite{Zambella:90a} and Terwijn \cite{Terwijn:98}.

Some of the main results on Demuth randomness and lowness notions are as follows:
\begin{itemize}
\item[(i)] \Kuc\ and Nies  \cite{Kucera.Nies:11} proved that every c.e.\ set Turing below a Demuth random is strongly jump traceable. Greenberg and Turetsky \cite{Greenberg.Turetsky:14} have recently  provided a converse of this  result: every  c.e.\ strongly jump traceable set has  a Demuth random set Turing above. 

\medskip
 
\item[(ii)] Nies~\cite{Nies:11} showed that  each base  for Demuth randomness is strongly jump traceable. Greenberg and Turetsky \cite{Greenberg.Turetsky:14} proved that this inclusion is proper.

\medskip

\item[(iii)] Lowness for Demuth randomness and weak Demuth randomness have been characterized by Bienvenu et al.\  \cite{Bienvenu.Downey.ea:nd}. The former is given by a notion called BLR-traceability (first defined by Cole and Simpson in \cite{Cole.Simpson:07}), in  conjunction with being computably dominated. The latter  is the same as  being computable. 
\end{itemize}
	
\section{Randomness, semigenericity, and $\mathrm{tt}$-reducibility}\label{sec-rstt}

In this last section, we discuss Demuth's work published near the end of his life, 
namely his work on $\mathrm{tt}$-reducibility in \cite{Demuth:88} and \cite{Demuth:88a} and his work on semigenericity in   \cite{Demuth:87} and \cite{Demuth.Kucera:87}, the latter paper written jointly with \Kuc. 

\subsection{Reducibilities from constructive analysis}\label{subsec-reducibilities}

In \cite{Demuth:88a}, Demuth proved a number of results connecting truth-table reducibility and various reducibilities from constructive analysis.  These results can be seen as providing bridge principles between certain concepts from computability theory and concepts from constructive analysis.

Recall from $\S$\ref{sec-Markov} that  the operator $R$ maps a Markov computable function $g$ to the maximal continuous extension $R[g]$ of $g$.
Using this operator, Demuth defines the following reduction for pairs of reals.

\begin{definition}
Given $\alpha,\beta\in[0,1]$, $\alpha$ is \emph{$f$-reducible} to $\beta$, denoted $\alpha\leq_f\beta$, if there is a Markov computable function $g$ such that  
 \[
R[g](\beta)=\alpha.
\]
In this case, we say that $\alpha$ is $f$-reducible to $\beta$ \emph{via $g$}.
\end{definition}

The relation $\leq_f$ is transitive.  This follows from the fact that for any Markov computable functions $g_1,g_2$, $R[g_1\circ g_2]=R[g_1]\circ R[g_2]$, which can be routinely verified.

Even though $\alpha$ and $\beta$ may be highly non-constructive reals, the reduction from $\alpha$ to $\beta$ is in a sense constructively grounded, being witnessed by the extension of a Markov computable function.  

Demuth then introduces three variants of $f$-reducibility:

\begin{definition}
\begin{enumerate}
\item $\alpha$ is \emph{$\emptyset$-$\mathrm{ucf}$-reducible}
to $\beta$, denoted $\alpha\leq_{\ES\text{-}\mathrm{ucf}}\beta$, if $\alpha$ is $f$-reducible to $\beta$ via a Markov computable function $g$ that is $\emptyset$-uniformly continuous.

\smallskip

\item $\alpha$ is  \emph{$\emptyset'$-$\mathrm{ucf}$-reducible}
to $\beta$, denoted $\alpha\leq_{\ES'\text{-}\mathrm{ucf}}\beta$, if $\alpha$ is $f$-reducible to $\beta$ via a Markov computable function $g$ that is $\emptyset'$-uniformly continuous.

\smallskip

\item $\alpha$ is  \emph{$\mathrm{mf}$-reducible}
to $\beta$, denoted $\alpha\leq_{\mathrm{mf}}\beta$, if $\alpha$ is $f$-reducible to $\beta$ via a Markov computable function $g$ that is monotonically increasing.
\end{enumerate}
\end{definition}

In order to compare these reducibilities to those from classical computability theory, Demuth identifies infinite sequences in $2^\NN$ with reals in [0,1].  Further, Demuth excludes a set $\mathcal{C}\subseteq 2^{\NN}$ of $\emptyset$-measure zero that contains all finite and cofinite sequences.

%
%\begin{definition}
%\begin{itemize}
%\item[(i)] A c.e.\ prefix-free set $S\subseteq 2^{<\NN}$ is a \emph{finite set cover} if for every non-empty finite set $Z$, one of its two binary representations is covered by $\llb S\rrb$.
%\item[(ii)] $Z\in 2^\omega$ is \emph{strongly bi-infinite} if $Z$ is bi-infinite and there is some finite set cover $S$ such that $Z\notin\llb S\rrb$.
%\item[(iii)] $Z\in2^\omega$ is \emph{weakly bi-infinite} if $Z$ is bi-infinite but not strongly bi-infinite.
%\end{itemize}
%\end{definition}

Demuth then proves the following:

\begin{theorem}[Demuth \cite{Demuth:88a}]\label{thm-tt-ucf}

\begin{enumerate}
\item For any $\emptyset$-uniformly continuous Markov computable function $f$, one can uniformly obtain an index of a $\mathrm{tt}$-functional $\Phi$ such that for every $A,B\in2^\NN$ such that $B\notin\mathcal{C}$,
\[
A\leq_{\ES\text{-}\mathrm{ucf}}B\;\text{via}\;f\;\text{if and only if}\;A\leq_{\mathrm{tt}}B\;\text{via}\;\Phi.
\]
\item For any $\mathrm{tt}$-functional $\Phi$, one can uniformly obtain the index of a $\emptyset$-uniformly continuous Markov computable function $f:[0,1]\rightarrow[0,1]$ such that for any $A,B\in 2^\NN$ such that $A,B\notin\mathcal{C}$
\[
A\leq_{\ES\text{-}\mathrm{ucf}}B\;\text{via}\;f\;\text{if and only if}\;A\leq_{\mathrm{tt}}B\;\text{via}\;\Phi.
\]
\end{enumerate}
\end{theorem}

Demuth also proved that Theorem \ref{thm-tt-ucf} can be relativized to $\ES'$ by using the notions of $\ES'$-$\mathrm{ucf}$-reducibility and $\ES'$-$\mathrm{tt}$-reducibility, excluding sequences from a set $\widehat{\mathcal{C}}\subseteq2^{\NN}$ of $\ES'$-measure zero.  In addition, Demuth proved similar results for $\mathrm{tt}$-reducibility and $\mathrm{mf}$-reducibility (see Theorem 13 and 14 of \cite{Demuth:88a}).

These theorems relating $\mathrm{tt}$-reducibility and the reducibilities from constructive analysis were essentially used in Demuth's proof of a theorem on the behavior of Martin-L\"of random reals under $\mathrm{tt}$-reducibility, to which we now turn.

\subsection{Truth-table reductions to random sequences}

The following is one of the most well-studied of Demuth's results (for instance, in \cite{Kautz:91}), which is referred to as ``Demuth's Theorem" in \cite{Downey.Hirschfeldt:book}.  We formulate the result here in terms of members of $2^\NN$.

\begin{theorem}[Demuth \cite{Demuth:88}]\label{thm:demuth}
If $B$ is non-computable and $\mathrm{tt}$-reducible to a Martin-L\"of random $A$, then there is a Martin-L\"of random $C$ such that
\[
B\leq_{\mathrm{tt}}C\leq_T B.
\]
\end{theorem}
Following Kautz's reconstruction in \cite{Kautz:91}, in which he only proves that $B\equiv_T C$, recent proofs of this result are given in terms of computable measures (see \cite[Section 8.6]{Downey.Hirschfeldt:book}).  A measure $\mu$ on $2^\NN$ is computable if $\mu(\Opcl\sigma)$ is a computable real uniformly in $\sigma\in2^{<\NN}$.

While the standard definition of Martin-L\"of randomness is formulated in terms of the Lebesgue measure, for any computable measure $\mu$  one can also define Martin-L\"of randomness with respect to $\mu$ simply by replacing the condition $\lambda(U_i)\leq 2^{-i}$ with $\mu(U_i)\leq 2^{-i}$ for each Martin-L\"of test $(U_n)_{n\in\NN}$.

Kautz recognized that Demuth's result follows from several facts about randomness and measures.  First, for any $\mathrm{tt}$-functional $\Phi$, the measure $\lambda_\Phi$ defined by $\lambda_\Phi(\Opcl\sigma)=\lambda(\Phi^{-1}(\Opcl\sigma))$ is a computable measure.  One can show this using Nerode's characterization of $\mathrm{tt}$-functionals as total Turing functionals (see \cite{Nerode:57a}).  Second, for any $\mathrm{tt}$-functional $\Phi$ and any Martin-L\"of random $A$, $\Phi(A)$ is Martin-L\"of random with respect to $\lambda_\Phi$ (a result due to Levin \cite{Levin.Zvonkin:70}).
Third, as shown by Kautz (and independently and earlier by Levin; see \cite{Levin.Zvonkin:70}), for any computable measure $\mu$, if $A$ is Martin-L\"of random with respect to $\mu$ and is not computable, then there is a Martin-L\"of random $B$ (with respect to $\lambda$) such that $A\equiv_T B$.

On the surface, Demuth's proof of his result takes a very different approach.  A rough sketch of his proof is as follows.  First, Demuth applies part 2 of Theorem \ref{thm-tt-ucf} from the previous section to replace the initial $\mathrm{tt}$-reduction $\Phi$ with an $\ES$-$\mathrm{ucf}$ reduction from $B$ to $A$ given by some Markov computable function $f$.  From this function $f$, Demuth then defines a monotone Markov computable function $g$, which allows him to construct (effectively in $B$) the set $C$ and an $\mathrm{mf}$-reduction from $B$ to $C$.  Lastly, by part 1 of Theorem \ref{thm-tt-ucf}, this $\mathrm{mf}$-reduction yields the desired $\mathrm{tt}$-reduction from $B$ to $C$.

Close examination of Demuth's proof shows that the function $g$ witnessing the $\mathrm{mf}$-reduction in his proof is the distribution function of the computable measure induced by the initial $\mathrm{tt}$-functional $\Phi$.  The use of distribution functions is at the heart of Kautz's proof, which shows that Demuth's proof is not too dissimilar from Kautz's reconstruction.

%In the 1980s the mathematics department at Charles  University had a seminar  on recursion theory, which was  based on Rogers' book \cite{Rogers:67} and some draft of Soare's book \cite{Soare:87}.  Because of this,  Demuth became more interested in computability theory and  the computational complexity of random sets.  
%
%\subsection{Randomness and computational complexity}
% Demuth proved the following. 
% 
% \begin{thm} (i) Each Demuth random real $z$ satisfies   $z' \le z \oplus \Halt$. 
% 
% %
% \n  (ii) Each Demuth random set is of hyperimmune $T$-degree.  \end{thm}  
% 
%\n (i). Demuth \cite[Remark 10, part 3b]{Demuth:88}  gives  a sketch of a proof. As  mentioned, a   full proof   can be found in \cite[3.6.26]{Nies:book}.
%
%\n (ii). Only a sketch of a proof is given in Remark 2 and Remark 11 of  the preprint survey. It seems that a single Demuth test is sufficient here.  
%  An alternative  proof  can be derived from (i)  and  the result of Miller and Nies \cite[Thm.\ 8.1.19]{Nies:book} that no   $GL_1$ set of hyperimmune-free degree is   d.n.c. 
%	(d.n.c.  stands for  diagonally noncomputable).

Demuth's result is, in a sense, the best possible.  One might hope to improve the theorem by showing the existence of a Martin-L\"of random $C$ such that $B\leq_{\mathrm{tt}}C\leq_{\mathrm{wtt}}B$, or even $B\equiv_{\mathrm{tt}} C$.  But this cannot be achieved, as shown by the following theorem.

\begin{theorem}[Bienvenu, Porter \cite{Bienvenu.Porter:12}]\label{thm:demuth-wtt}
There is a Martin-L\"of random $A$ and a $\mathrm{tt}$-functional $\Phi$ such that $\Phi(A)$ is non-computable and cannot $\mathrm{wtt}$-compute any Martin-L\"of random.
\end{theorem}

In this same paper \cite{Bienvenu.Porter:12}, using the technique discussed above, the following result was shown without the authors being aware that Demuth had already proved it.

\begin{theorem}[Demuth \cite{Demuth:87}]
There is a $\mathrm{tt}$-degree containing both a c.e.\ set and a Martin-L\"of random set. Thus there is some c.e.\ set $S\in2^\NN$ that is Martin-L\"of random with respect to some computable measure.
\end{theorem}

\subsection{Semigenericity}\label{subsec-semigenericity}

Researchers in algorithmic randomness are interested in the relationship between notions of effective randomness and effective genericity.  Demuth too was interested in this relationship, studying a notion he referred to as semigenericity.  

\begin{definition} [\cite{Demuth:87}]
A non-computable set $Z$ is called \emph{semigeneric}  
%\cite{Demuth:87} 
if every $\PI 1$ class containing $Z$ has a computable member.
\end{definition}
Intuitively, to be semigeneric means to be close to computable in the sense that the set cannot be separated from the computable sets by a $\PI 1$ class.

The notion of semigenericity was   studied independently  though much later in Joseph Miller's thesis \cite{Miller:02}, who referred to the notion as \emph{unavoidability}, although Miller also counted the computable points as unavoidable.  As noted in \cite{Miller:02}, non-computable unavoidable points were also studied by Kalantari and Welch in \cite{Kalantari.Welch:2003}, who referred to these points as \emph{shadow points}.

It is particularly natural to study semigenericity from the point of view of  constructive mathematical analysis.  As discussed in $\S$\ref{sec-Markov}, one can define a Markov computable function in terms of a $\SI 1$ class $\mathcal A$ that contains every computable set. Since the complement of $\mathcal A$ is a $\Pi^0_1$ class with no computable members, it follows that $\mathcal A$ contains every semigeneric real.  From this fact, one can show that for every Markov computable function $g$, the classical extension $R[g]$ of $g$ is continuous at every semigeneric real.

Demuth and \Kuc{}   \cite{Demuth.Kucera:87} studied  semigenericity  and its 
relationship with other types of  genericity.  
We review some of their results. First, Demuth and \Kuc{} showed that semigenericity is closely related to a notion studied by Ce\v\i tin.

\begin{definition} [\cite{Ceitin:70}]
 A set $Z$ is called \emph{strongly undecidable} if there is a partial computable function $p$ such that for any computable set $M$ and any index $v$ of its characteristic function,  
$p(v)$  is defined and $Z \uhr {p(v)} \neq  M \uhr {p(v)} $.
\end{definition}

% Ce{\u{\i}}tin \cite{Ceitin:70} called a  set $Z$  \emph{strongly undecidable} if there is a computable function $p$ such that for any computable set M and any index $v$ of its characteristic function,  
%$p(v)$  is defined and $Z \uhr {p(v)} \neq  M \uhr {p(v)} $.

%For non-computable sets non-semigenericity easily implies strong undecidability. For the other implication, 
%given a strongly undecidable set $Z$ we can construct, by using Recursion Theorem, a $\Sigma^0_1$ class containing all computable sets but not $Z$ 
%(this implication can also be derived from a result of Kushner ([\textcolor{green}{see below}]). Worthwhile ??
%
%
%
%\n \textcolor{green}{
%Kushner, B. A., Coverings of separable sets, Studies in the theory of Algorithms and Mathematical Logic,
%Vol. I, Vychisl. Tsentr. Akad. Nauk, SSSR, Moscow, 1973, pp.235-246.(Russian)
%}

In \cite[Cor.\ 2]{Demuth.Kucera:87}, Demuth and \Kuc{} proved that a non-computable   set $Z$ is semigeneric if and only if $Z$ is not  strongly undecidable.  This same result was also obtained by Miller (see \cite[Proposition 4.2.4]{Miller:02}).  Miller also studied a variant of strong undecidability in which one requires that the function $p$ be total; he referred to this stronger notion as \emph{hyperavoidability}.  Interestingly, the hyperavoidable reals were recently shown by Kjos-Hanssen, Merkle, and Stephan to be equivalent to an important class in algorithmic randomness known as the \emph{complex reals}.  A non-dyadic rational $x\in[0,1]$ is complex if for the sequence $X\in2^\NN$ such that $x=0.X$, there is some computable order such that~${C(X\uh n)\geq f(n)}$ for all $n$.  Here $C(\sigma)$ is the plain Kolmogorov complexity of $\sigma\in\{0,1\}^*$.  See \cite[Section 3]{Kjos.ea:2011} for more details.

Demuth and \Kuc{} also proved that strong undecidability  can be characterized by some kind of ``uniform non-hyperimmunity'': by \cite[Thm.\ 5]{Demuth.Kucera:87}, a set $Z$ is strongly undecidable 
if and only if there is a computable function $f$ such that for each computable set $M$  and any index $v$ of its characteristic function, 
the symmetric difference  $M \triangle Z$
is infinite  and  its listing in order of magnitude is
dominated by the computable function with index $f(v)$.

In \cite[Thm.\  14]{Demuth.Kucera:87}, Demuth and \Kuc{}  also   characterized  the  sets $Z$ such that the Turing degree of $Z$ contains 
 a strongly undecidable set:   this happens precisely when there is a $\PI 2$ class containing $Z$ but no computable sets.   Thus   we have a weaker form of  separation from the computable sets than for non-computable sets that are not semigeneric, where the separating class is  $\PI 1$ by definition.

 This result was actually proved in terms of so-called $V$-coverings (where $V$ stands for 
Vitali).
A set $Z$ is \emph{$V$-covered} by a 
c.e. set of strings $A$   if for all $k$  there is a string $\sigma \in A$ such that $|\sigma| \geq k$ and $\sigma \prec Z$.  It is easy to see that a class of sets $\mathcal A$ is a $\Pi^0_2$ class if and only if there is a c.e. set of strings $B$ such that
$\mathcal A$ is equal to the class of sets  $V$-covered by $B$ (see \cite[Proposition 1.8.60]{Nies:book}).

Demuth also studied the relationship between semigenericity and weak 1-genericity, which was introduced by Kurtz in \cite{Kurtz:81}.  Recall   that a set $Z$ is weakly $1$-generic if $Z$ is in each dense  $\Sigma_1^0$ class.  Clearly any weakly $1$-generic set is semigeneric. 
However, the converse fails; for instance, Demuth proved in \cite[Theorem 9.2]{Demuth:87} that if $Z$ is weakly 1-generic, then $Z\oplus Z$ is semigeneric but not weakly 1-generic.

We conclude this section with a discussion of a number of results that Demuth obtained on the relationship between randomness and semigenericity.  As shown by Demuth, one immediate consequence of the definition of semigenericity is that no Martin-L\"of random is semigeneric, since every Martin-L\"of random is contained in a $\Pi^0_1$ class with no computable members, given by the complement of some finite level of the universal Martin-L\"of test.  We should note, however, that there is a computable measure $\mu$ such that some Martin-L\"of random sequence with respect to $\mu$ is semigeneric.  For instance, the real $\Phi(A)$ in the statement of Theorem \ref{thm:demuth-wtt} is Martin-L\"of random with respect to the induced measure $\lambda_\Phi$ and is also semigeneric.

One interesting similiarity between Martin-L\"of randomness and semigenericity is the following. Demuth's Theorem \ref{thm:demuth} implies that the class of non-computable reals that are Martin-L\"of random with respect to some computable measure is closed downwards under $\mathrm{tt}$-reducibility.  Similarly, semigenericity is closed downwards under $\mathrm{tt}$-reducibiilty:  Demuth  proved in \cite[Thm.\ 9]{Demuth:87} that  if a set $Z$ is semigeneric then any set $B$ such that $\emptyset <_{\mathrm{tt}} B \leq_{\mathrm{tt}} Z$ is also  semigeneric.
In particular, its   $\mathrm{tt}$-degree only contains  semigeneric sets.

Demuth also proved that no semigeneric real can $\mathrm{tt}$-compute a Martin-L\"of random real.  The example from Theorem \ref{thm:demuth-wtt} shows that the converse does not hold:  $\Phi(A)$ is semigeneric and $\mathrm{tt}$-reducible to the Martin-L\"of random real $A$.  On a similar note, Demuth and \Kuc{} also proved that no 1-generic (a notion slightly stronger than weak 1-genericity) can compute a Martin-L\"of random (see \cite[Corollary 9 of Theorem 8]{Demuth.Kucera:87}).

A number of connections between semigenericity and Denjoy randoms were also established by Demuth in \cite{Demuth:90}.  As discussed in $\S$\ref{subsec-denjoy-alternative}, Demuth proved that every non-Martin-L\"of Denjoy random real is high but not the converse result that every high degree contains such a real (as proved in \cite{Nies.Stephan.ea:05}).  However, he was close, as he showed that every real of high degree can compute a semigeneric Denjoy random real.
In this same work, Demuth proved that there is a minimal Turing degree containing a semigeneric Denjoy random real, and that
every semigeneric Denjoy random real is $\mathrm{tt}$-reducible to a Denjoy random that is neither semigeneric nor Martin-L\"of random.

%\vsps
%
%
%%{\n \emph{Connection to weak 1-genericity and hyperimmunity.}} \\
%\n {\large \emph{Connection to some other computability-theoretic notions}}
%
%\vsps

%
%Demuth  \cite[Thm.\ 16]{Demuth:87}  showed that a  set $Z$ is weakly $1$-generic if and only if for any computable set $M$ the symmetric difference $M \triangle Z$ is hyperimmune. 
%Kurtz~\cite{Kurtz:81,Kurtz:83} proved  that a Turing-degree contains a weakly $1$-generic set if and only if it is hyperimmune. 
%It follows from Kurtz's results,  using a fact of Martin-Miller~\cite{Martin.Miller:68},  that the  weakly $1$-generic $T$-degrees are closed upwards.
%As a corollary we have that there are weakly $1$-generic Turing degrees which do contain ML-random sets and, thus, 
%they can compute d.n.c. functions. On the other hand, \Kuc{}  and Demuth showed that  the classes of $1$-generic Turing degrees 
%and of Turing degrees of d.n.c. functions are disjoint. In fact, they proved  in \cite[Cor.\ 2]{Demuth.Kucera:87} (also see \cite[4.1.6]{Nies:book})   the following,     .

\medskip

\section{Concluding remarks}\label{sec-concluding}
  
As we have seen, Demuth's contribution to the study of constructive mathematics, and in particular, his work on the various definitions of randomness in the context of constructive analysis, is remarkable in the depth and breadth of ideas that it contains.
Despite working largely in isolation, Demuth produced an enormous number of results, some of which have been subsequently rediscovered, and some of which have yet to be fully understood.  As our discussion has shown, many recent developments in algorithmic randomness can be seen as extending Demuth's larger project of bringing the tools of computability theory to bear on the study of constructive analysis.

The searchable  database at 
%\vspace{-.02in}
%\[
%\href{http://cca-net.de/publications/unclassified.html}{\texttt{http://cca-net.de/publications/unclassified.html}} 
%\]
%provides references for most of Demuth's papers, and  the searchable  database at 
\[
\texttt{http://www.dml.cz}
\]contains papers of Demuth 
published in Commentationes Mathematicae Universitatis Carolinae (CMUC) or Acta Universitatis Carolinae (AUC).

\begin{figure}[hbt] 
\bc \scalebox{1.34}
{\includegraphics{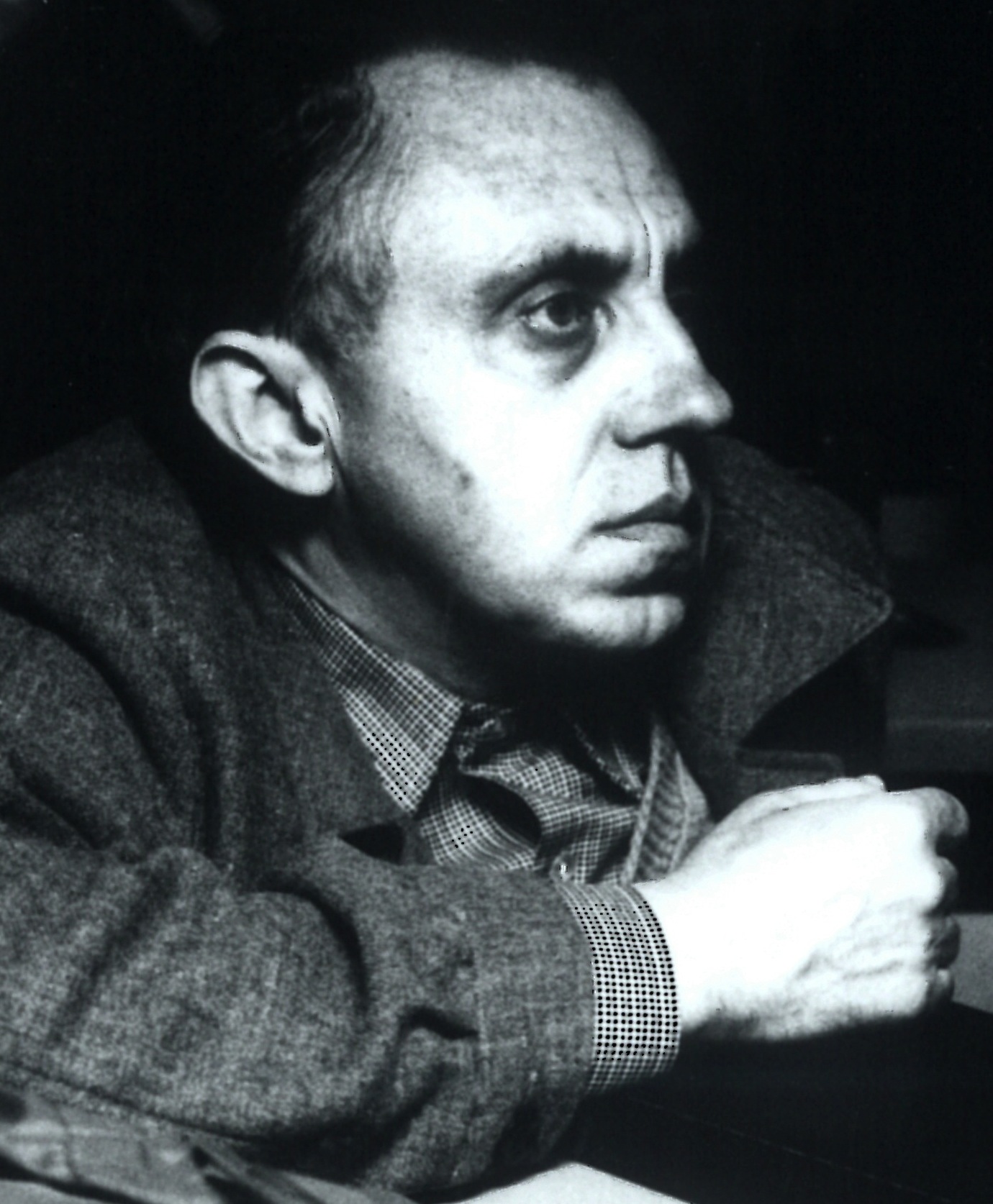}} 
%{\includegraphics{figures2011/Demuth_foto_1_small.jpg}} 
\ec
	\caption{Osvald Demuth}
\end{figure}
 
\vspace*{-1ex}

	%The searchable  database at \href{http://www.dml.cz}{\texttt{http://www.dml.cz}} contains most papers of Demuth. 
  
%\newpage

 \bibliographystyle{asl}
\bibliography{./../../Logicsharing/bibs/Nies,./../../Logicsharing/bibs/randomness,./../../Logicsharing/bibs/various,./../../Logicsharing/bibs/recursiontheory,./../../Logicsharing/bibs/Kucera,./../../Logicsharing/bibs/analysis}
%
%\bibliography{./../../../Logicsharing/bibs/Nies,./../../../Logicsharing/bibs/randomness,./../../../Logicsharing/bibs/various,./../../../Logicsharing/bibs/recursiontheory,./../../../Logicsharing/bibs/Kucera,./../../../Logicsharing/bibs/analysis}

\end{document}